\documentclass[a4paper,12pt]{preprint}
\usepackage[margin=1.3in]{geometry}  
\usepackage[full]{textcomp}
\usepackage[osf]{newtxtext}

\usepackage{mathalfa}
\usepackage{microtype}

\usepackage{booktabs}
\usepackage{times}

\usepackage{enumitem}
\usepackage{amsmath}
\usepackage{amsfonts} 
\usepackage{amssymb}             

\usepackage{comment}
\usepackage{breakurl}

\usepackage{mathrsfs}
\usepackage{booktabs}
\usepackage{mathtools}

\usepackage{tikz}
\usepackage{amsthm}                
\usepackage{mhequ}
\usepackage{hyperref}

\hypersetup{pdfstartview=}

\addtocontents{toc}{\protect\hypersetup{hidelinks}}

\DeclareSymbolFont{sfoperators}{OT1}{ptm}{m}{n}
\DeclareSymbolFontAlphabet{\mathsf}{sfoperators}

\makeatletter
\def\operator@font{\mathgroup\symsfoperators}
\makeatother

\numberwithin{equation}{section}

\newtheorem{thm}{Theorem}[section]
\newtheorem{defn}[thm]{Definition}
\newtheorem{lem}[thm]{Lemma}
\newtheorem{prop}[thm]{Proposition}

\newtheorem{assumption}[thm]{Assumption}

\makeatletter
\def\th@newremark{\th@remark\thm@headfont{\bfseries}}

\def\bdiamond{\mathop{\mathpalette\bdi@mond\relax}}
\newcommand\bdi@mond[2]{%
	\vcenter{\hbox{\m@th
			\scalebox{\ifx#1\displaystyle 2.6\else1.8\fi}{$#1\diamond$}%
	}}%
}

\def\bDiamond{\mathop{\mathpalette\bDi@mond\relax}}
\newcommand\bDi@mond[2]{%
	\vcenter{\hbox{\m@th
			\scalebox{\ifx#1\displaystyle 2.6\else1.2\fi}{$#1\Diamond$}%
	}}%
}
\makeatletter

\theoremstyle{newremark}

\newtheorem{rmk}[thm]{Remark}
\newtheorem{eg}[thm]{Example}

\definecolor{darkgreen}{rgb}{0.1,0.7,0.1}
\definecolor{darkred}{rgb}{0.7,0.1,0.1}
\definecolor{darkblue}{rgb}{0,0,0.7}
\newcommand{\EE}{\mathbb{E}}     

\newcommand{\VV}{\mathbb{V}}

\newcommand{\aA}{\mathcal{A}}

\newcommand{\cC}{\mathcal{C}}

\newcommand{\hH}{\mathcal{H}}
\newcommand{\iI}{\mathcal{I}}

\newcommand{\kK}{\mathcal{K}}
\newcommand{\lL}{\mathcal{L}}

\newcommand{\nN}{\mathcal{N}}
\newcommand{\oO}{\mathcal{O}}

\newcommand{\sS}{\mathcal{S}}

\newcommand{\uU}{\mathcal{U}}

\newcommand{\fs}{\mathfrak{s}}

\newcommand{\fv}{\mathfrak{v}}

\makeatletter 
\newcommand{\cov}{{\operator@font cov}}
\newcommand{\var}{{\operator@font var}}
\newcommand{\corr}{{\operator@font corr}}
\newcommand{\diam}{{\operator@font diam}}
\newcommand{\Av}{{\operator@font Av}}
\newcommand{\trig}{{\operator@font trig}}
\newcommand{\Enh}{{\operator@font Enh}}
\newcommand{\EEnh}{\overline {\operator@font Enh}}
\makeatother

\newcommand{\lfl}{\left\lfloor }  
\newcommand{\rfl}{\right\rfloor} 

\newcommand{\E}{\mathbf{E}}

\newcommand{\N}{\mathbf{N}}
\newcommand{\R}{\mathbf{R}}

\newcommand{\X}{\mathbf{X}}

\newcommand{\Z}{\mathbf{Z}}

\renewcommand{\k}{\mathbf{k}}

\newcommand{\n}{\mathbf{n}}

\newcommand{\x}{\mathbf{x}}

\newcommand*\Bell{\ensuremath{\boldsymbol\ell}}
\newcommand*\Bbeta{\ensuremath{\boldsymbol\beta}}

\def\set{{\mathfrak{u}}}

\newcommand{\hF}{\widehat{F}}

\def\Clus{\mathscr{C}}

\newcommand{\eps}{\varepsilon}
\newcommand{\Ups}{\Upsilon}

\colorlet{symbols}{blue!90!black}
\colorlet{testcolor}{green!60!black}

\def\${|\!|\!|}

\usetikzlibrary{shapes.misc}
\usetikzlibrary{shapes.symbols}
\usetikzlibrary{snakes}
\usetikzlibrary{decorations}
\usetikzlibrary{decorations.markings}

\def\drawx{\draw[-,solid] (-3pt,-3pt) -- (3pt,3pt);\draw[-,solid] (-3pt,3pt) -- (3pt,-3pt);}
\tikzset{
	root/.style={circle,fill=testcolor,inner sep=0pt, minimum size=2mm},
	dot/.style={circle,fill=black,inner sep=0pt, minimum size=1mm},
	edot/.style={circle,fill=black,inner sep=0pt, minimum size=1mm},
	odot/.style={circle,draw=black,inner sep=0pt, minimum size=1mm},
	var/.style={circle,fill=black!10,draw=black,inner sep=0pt, minimum size=
	2mm},
    svar/.style={circle,fill=black!10,draw=black,inner sep=0pt, minimum size=
	1.5mm},
    noise0/.style={rectangle,draw=symbols,fill=white,inner sep=0pt, minimum size=1.5mm},
    noise1/.style={circle,draw=symbols,fill=white,inner sep=0pt, minimum size=1.5mm},
    noise2/.style={circle,draw=symbols,fill=symbols,inner sep=0pt, minimum size=1.5mm},
	dotred/.style={circle,fill=symbols!50,inner sep=0pt, minimum size=2mm},
	generic/.style={semithick,shorten >=1pt,shorten <=1pt},
	ageneric/.style={semithick},
	dist/.style={ultra thick,draw=testcolor,shorten >=1pt,shorten <=1pt},
	testfcn/.style={ultra thick,testcolor,shorten >=1pt,shorten <=1pt,<-},
	testfcnx/.style={ultra thick,testcolor,shorten >=1pt,shorten <=1pt,<-,
		postaction={decorate,decoration={markings,mark=at position 0.6 with {\drawx}}}},
	kepsilon/.style={semithick,shorten >=1pt,shorten <=1pt,densely dashed,->},
	kprimex/.style={semithick,shorten >=1pt,shorten <=1pt,densely dashed,->,
		postaction={decorate,decoration={markings,mark=at position 0.4 with {\drawx}}}},
	kernel/.style={semithick,shorten >=1pt,shorten <=1pt,->},
	akernel/.style={semithick,->},
	multx/.style={shorten >=1pt,shorten <=1pt,
		postaction={decorate,decoration={markings,mark=at position 0.5 with {\drawx}}}},
	kernelx/.style={semithick,shorten >=1pt,shorten <=1pt,->,
		postaction={decorate,decoration={markings,mark=at position 0.4 with {\drawx}}}},
	kernel1/.style={->,semithick,shorten >=1pt,shorten <=1pt,postaction={decorate,decoration={markings,mark=at position 0.45 with {\draw[-] (0,-0.1) -- (0,0.1);}}}},
	kernel2/.style={->,semithick,shorten >=1pt,shorten <=1pt,postaction={decorate,decoration={markings,mark=at position 0.45 with {\draw[-] (0.05,-0.1) -- (0.05,0.1);\draw[-] (-0.05,-0.1) -- (-0.05,0.1);}}}},
	kernelBig/.style={semithick,shorten >=1pt,shorten <=1pt,decorate, decoration={zigzag,amplitude=1.5pt,segment length = 3pt,pre length=2pt,post length=2pt}},
	gepsilon/.style={dotted,semithick,shorten >=1pt,shorten <=1pt},
	renorm/.style={shape=circle,fill=white,inner sep=1pt},
	labl/.style={shape=rectangle,fill=white,inner sep=1pt},
	xi/.style={circle,fill=symbols!10,draw=symbols,inner sep=0pt,minimum size=1.2mm},
	xix/.style={crosscircle,fill=symbols!10,draw=symbols,inner sep=0pt,minimum size=1.2mm},
	xib/.style={circle,fill=symbols!10,draw=symbols,inner sep=0pt,minimum size=1.6mm},
	xibx/.style={crosscircle,fill=symbols!10,draw=symbols,inner sep=0pt,minimum size=1.6mm},
	not/.style={circle,fill=symbols,draw=symbols,inner sep=0pt,minimum size=0.5mm},
	>=stealth,
  	highlight/.style={line width=7pt,blue,draw opacity=0.2,line cap=round,line join=round},
  	cover/.style={line width=7pt,blue,line cap=round,line join=round},
	smalldot/.style={circle,fill=symbols,draw=symbols, solid,inner sep=0pt,minimum size=0.5mm},
	}

\makeatletter
\def\DeclareSymbol#1#2#3{\expandafter\gdef\csname MH@symb@#1\endcsname{\tikz[baseline=#2,scale=0.15,draw=symbols]{#3}}\expandafter\gdef\csname MH@symb@#1s\endcsname{\scalebox{0.5}{\tikz[baseline=#2,scale=0.15,draw=symbols]{#3}}}}
\def\<#1>{\csname MH@symb@#1\endcsname}
\makeatother

\DeclareSymbol{Xi22}{0.5}{\draw (0,0) node[xi] {} -- (-1,1) node[not] {} -- (0,2) node[xi] {}; }

\DeclareSymbol{0'}{0}{\node at (0,0.7) [noise0] {}; }
\DeclareSymbol{1'}{0}{\node at (0,0.7) [noise1] {}; }
\DeclareSymbol{2'}{0}{\node at (0,0.7) [noise2] {}; }

\DeclareSymbol{0'0}{0}{\draw (0,-0.5) node[not] {} -- (0,1.5) node[noise0] {};}
\DeclareSymbol{1'0}{0}{\draw (0,-0.5) node[not] {} -- (0,1.5) node[noise1] {};}
\DeclareSymbol{2'0}{0}{\draw (0,-0.5) node[not] {} -- (0,1.5) node[noise2] {};}
\DeclareSymbol{1'1'}{0}{\draw (-1,1.5) node[noise1] {}  -- (0.5,0) node[noise1] {};}
\DeclareSymbol{2'1'}{0}{\draw (-1,1.5) node[noise2] {}  -- (0.5,0) node[noise1] {};}
\DeclareSymbol{2'1'0}{0}{\draw (1.5,2) node[noise2] {}  -- (0,0.5) node[noise1] {} -- (1.5,-1) node[not] {};}
\DeclareSymbol{2'2'0}{-2}{\draw (-1,1) node[noise2] {}  -- (0,-0.5) node[smalldot] {} -- (1,1) node[noise2] {};}
\DeclareSymbol{2'1'1'}{-1}{\draw (0,1.4) node[noise2] {}  -- (1.5,0.5) node[noise1] {} -- (0,-0.4) node[noise1] {};}
\DeclareSymbol{2'0'}{0}{\draw (-1,1.5) node[noise2] {}  -- (0.5,0) node[noise0] {};}
\DeclareSymbol{2'2'0'}{0}{\draw (-1,1.5) node[noise2] {}  -- (0,0) node[noise0] {} -- (1,1.5) node[noise2] {};}

\DeclareSymbol{0}{0}{\draw[white] (-.4,0) -- (.4,0); \draw (0,0)  -- (0,1.2) node[smalldot] {};}
\DeclareSymbol{1}{0}{\draw[white] (-.4,0) -- (.4,0); \draw (0,0)  -- (0,1.2) node[smalldot] {};}
\DeclareSymbol{2}{0}{\draw (-0.5,1.2) node[smalldot] {} -- (0,0) -- (0.5,1.2) node[smalldot] {};}
\DeclareSymbol{11}{0}{\draw (0,1.8) node[smalldot] {} -- (-0.7,0.9) -- (0,0) -- (0.7,1) node[smalldot] {};}
\DeclareSymbol{10}{0}{\draw (0,1.8) node[smalldot] {} -- (-0.8,0.9) -- (0,0);}
\DeclareSymbol{21}{0.7}{\draw (-1,1.8) node[smalldot] {} -- (-0.5,0.9); \draw (0,1.8) node[smalldot] {} -- (-0.5,0.9) -- (0,0) -- (0.5,0.9) node[smalldot] {};}
\DeclareSymbol{20}{0.7}{\draw (-1,1.8) node[smalldot] {} -- (-0.5,0.9); \draw (0,1.8) node[smalldot] {} -- (-0.5,0.9) -- (-0.5,0);}
\DeclareSymbol{210}{1.1}{\draw (-0.5,2.4) node[smalldot] {} -- (-1,1.6);\draw (-1.5,2.4) node[smalldot] {} -- (-0.5,0.8); \draw (0,1.6) node[smalldot] {} -- (-0.5,0.8) -- (-0.5,0);}
\DeclareSymbol{211}{1.1}{\draw (-0.5,2.4) node[smalldot] {} -- (-1,1.6);\draw (-1.5,2.4) node[smalldot] {} -- (-0.5,0.8); \draw (0,1.6) node[smalldot] {} -- (-0.5,0.8) -- (0,0) -- (0.5,0.8) node[smalldot] {};}
\DeclareSymbol{22}{0.7}{\draw (-1.5,1.8) node[smalldot] {} -- (-1,0.9) -- (0,0) -- (1,0.9) -- (1.5,1.8) node[smalldot] {}; \draw (-0.5,1.8) node[smalldot] {} -- (-1,0.9);\draw (0.5,1.8) node[smalldot] {} -- (1,0.9);}

\setlist[itemize]{topsep=3pt,itemsep=1.5pt,parsep=0pt}

\def\scal#1{\langle#1\rangle}
\def\cent#1{\mathopen{{\langle\kern-0.3em\rangle}}#1\mathclose{{\langle\kern-0.3em\rangle}}}

\def\d{\partial}

\begin{document}

\title{Sharp convergence of nonlinear functionals of a class of Gaussian random fields}
\author{Weijun Xu}
\institute{University of Oxford, UK {/} New York University Shanghai, China.\\
	\email{weijunx@gmail.com}}

\maketitle

\begin{abstract}
	We present a self-contained proof of a uniform bound on multi-point correlations of trigonometric functions of a class of Gaussian random fields. It corresponds to a special case of the general situation considered in \cite{KPZ_general}, but with improved estimates. As a consequence, we establish convergence of a class of Gaussian fields composite with more general functions. These bounds and convergences are useful ingredients to establish weak universalities of several singular stochastic PDEs. 
\end{abstract}

\setcounter{tocdepth}{2}
\microtypesetup{protrusion=false}
\tableofcontents
\microtypesetup{protrusion=true}

\def\k{\mathbf{k}}

\section{Introduction}

\subsection{Motivation from weak universalities}

\label{sec:weak_universality}

The study of singular stochastic PDEs has received much attention recently, and powerful theories are being developed to enhance the general understanding of this area. We refer to the excellent surveys \cite{Martin_renorm_survey} and \cite{Gubinelli_ICM} and references therein for recent breakthroughs in the field. 

One of the motivations to study singular SPDEs is that many of them are expected to be universal objects in crossover regimes of their respective universality classes, a phenomenon known as \textit{weak universality}. One well-known example is the KPZ equation (\cite{KPZ86}), formally given by
\begin{equation*}
\d_t h = \d_x^2 h + \lambda (\d_x h)^2 + \xi, 
\end{equation*}
where $\xi$ is the one dimensional space-time white noise. The equation is only formal since it involves the square of a distribution. Nevertheless, the solution can be rigorously constructed in a few different ways, including the Cole-Hopf transform (\cite{BG97}), pathwise solutions via rough paths / regularity structures (\cite{HairerKPZ, rs_theory, RP_book}) or paracontrolled distributions (\cite{Reloaded}), or the notion of energy solution through a martingale problem (\cite{Energy, Energy_unique}). 

The KPZ equation is expected to be the universal model for weakly asymmetric interface growth at large scales. In \cite{HQ}, the authors considered continuous microscopic models of the type
\begin{equation} \label{eq:KPZ_micro}
\d_t \tilde{h} = \d_x^2 \tilde{h} + \sqrt{\eps} F(\d_x \tilde{h}) + \tilde{\xi}
\end{equation}
for any even polynomial $F$ and smooth stationary Gaussian random field $\tilde{\xi}$. The main result in \cite{HQ} is that there exists $C_{\eps} \rightarrow +\infty$ such that the rescaled and re-centered height function
\begin{equation*}
h_{\eps}(t,x) := \eps^{\frac{1}{2}} \tilde{h}(t/\eps^2, x/\eps) - C_{\eps}t
\end{equation*}
converges to the solution of the KPZ equation with
\begin{equation} \label{eq:lambda}
\lambda = \frac{1}{2} \E F''(\tilde{\Psi}), 
\end{equation}
where $\tilde{\Psi} = \d_x P * \tilde{\xi}$ and $P$ is the heat kernel. \cite{KPZ_general} extended the result to arbitrary even functions $F$ with sufficient regularity and polynomial growth. Similar results have also been obtained in \cite{HQ_stationary} for models at stationarity. 

To see why the convergence holds with $\lambda$ given by \eqref{eq:lambda}, we write down the equation for $h_{\eps}$: 
\begin{equation*}
\d_t h_{\eps} = \d_x^2 h_{\eps} + \eps^{-1} F(\sqrt{\eps} \d_x h_\eps) + \xi_{\eps} - C_{\eps}, 
\end{equation*}
where $\xi_{\eps}(t,x) = \eps^{-\frac{3}{2}} \tilde{\xi}(t/\eps^2, x/\eps)$ approximates the space-time white noise $\xi$ at scale $\eps$. Let $\Psi_\eps = \d_x P * \xi_\eps$, then $\sqrt{\eps} \Psi_\eps$ is stationary Gaussian with finite variance. In addition, by analogy with the standard KPZ equation, it is reasonable to expect that the remainder $\d_x u_\eps =\d_x h_\eps - \Psi_\eps$ is almost bounded. Hence one can Taylor expand the nonlinearity $\eps^{-1} F(\sqrt{\eps} \d_x h_\eps)$ around $\sqrt{\eps} \Psi_\eps$, and formally get
\begin{equation*}
\begin{split}
\eps^{-1}F(\sqrt{\eps} \Psi_\eps) &- C_\eps =\big( \eps^{-1}F(\sqrt{\eps} \Psi_\eps) - C_{\eps} \big)\\
&+ \eps^{-\frac{1}{2}} F'(\sqrt{\eps} \Psi_\eps) \cdot (\d_x u_\eps) + F''(\sqrt{\eps} \Psi_\eps) \cdot (\d_x u_\eps)^{2} + \oO(\eps^{\frac{1}{2}-}). 
\end{split}
\end{equation*}
One then needs to show the convergence of the objects $\eps^{-1} F(\sqrt{\eps} \Psi_\eps) - C_\eps$, $\eps^{-\frac{1}{2}} F'(\sqrt{\eps} \Psi_\eps)$, $F''(\sqrt{\eps} \Psi_\eps)$ as well as their products, which arise from the local expansion of $\d_x u_\eps$. 

At least formally, by chaos expanding $\eps^{-1}F(\sqrt{\eps} \Psi_\eps)$ and taking $C_\eps = \eps^{-1} \E F(\sqrt{\eps} \Psi_\eps)$, one can see that
\begin{equation} \label{eq:converge_eg}
\eps^{-1} F(\sqrt{\eps} \Psi_\eps) - C_{\eps} \rightarrow \lambda \Psi^{\diamond 2}
\end{equation}
where $\lambda$ is given in \eqref{eq:lambda} and $\Psi = \d_x P * \xi$ is the limit of $\Psi_\eps$. This is because all terms starting from the $4$-th chaos vanish termwise as $\eps \rightarrow 0$, and only the second chaos component survives in the limit. 

When $F$ is even polynomial, this heuristic indeed gives a direct proof of the convergence of the term in \eqref{eq:converge_eg}. However, when $F$ is not polynomial, the actual proof of the convergence becomes much subtler. The main obstacle is that $F(\sqrt{\eps} \Psi_\eps)$ expands into an \textit{infinite} chaos series. If we brutally control their high moments termwise as in the polynomial case, then in order for these termwise moment bounds to be summable, we need to impose very strong conditions on $F$ (namely its Fourier transform being compactly supported), which is clearly too restrictive. 

Instead, in \cite{KPZ_general}, the authors expanded $F(\sqrt{\eps} \Psi_\eps)$ in terms of Fourier transform, developed a procedure in obtaining pointwise correlation bounds on trigonometric functions of Gaussians, and deduced the desired convergence from those bounds. 

Similar universality results are also present in the dynamical $\Phi^4_3$ model. The weak universality of $\Phi^4_3$ equation for a large class of symmetric phase coexistence models with polynomial potential was established in \cite{Phi4_poly} for Gaussian noise and then extended in \cite{Phi4_non_Gaussian} to non-Gaussian noise. The extension beyond polynomial potential (even with Gaussian noise) has the same difficulties as in the KPZ case discussed above. In the recent work \cite{Phi4_general}, the authors developed different methods based on Malliavin calculus to control similar objects. The methods developed in \cite{Phi4_general} and \cite{KPZ_general} to treat general nonlinearities are both robust enough to cover both KPZ and $\Phi^4_3$ equations as well as other similar situations. 

In this article, we follow the ideas developed in \cite{KPZ_general}, and prove a uniform bound in a special case considered in there. This special case is technically simpler to explain, but is also illustrative enough to reveal the main idea of the proof for the more general case. Furthermore, we obtain a better bound in this special case, thus yielding convergence results for functions $F$ with lower regularity. 

\subsection{Main statements}

Fix a scaling $\fs = (s_1, \dots, s_d)$ on $\R^d$, and let $|\fs| = \sum_{j} s_j$. The metric induced by $\fs$ is
\begin{equation*}
	|x|_{\fs} := |x_1|^{\frac{1}{s_1}} + \dots + |x_d|^{\frac{1}{s_d}}. 
\end{equation*}
Since the scaling is fixed throughout the article, we simply write $|x|$ instead of $|x|_{\fs}$. For any Gaussian random field $X$, any function $F: \R \rightarrow \R$ with at most exponential growth and any integer $m \geq 0$, we write
\begin{equation*}
\hH_{m}\big(F(X)\big) = \sum_{n \geq m} C_{n} X^{\diamond n}, 
\end{equation*}
where $X^{\diamond n}$ denotes the $n$-th Wick power of $X$, and $C_{n} = \frac{1}{n!} \E F^{(n)}(X)$ is the coefficient of $X^{\diamond n}$ in the chaos expansion of $F(X)$. In other words, $\hH_{m}\big(F(X)\big)$ is $F(X)$ with the first $m-1$ chaos removed. We refer to \cite[Chapter~1]{Nua06} for more details on chaos expansion of random variables. We have the following bound. 

\begin{thm} \label{th:main}
	Let $\alpha \in (0,|\fs|)$, and $\{\Phi_\eps\}_{\eps \in (0,1)}$ be a class of centered Gaussian random fields satisfying
	\begin{equation} \label{eq:field_eps}
	\frac{\eps^{\alpha}}{\Lambda \big(|x-y|+\eps\big)^{\alpha}} \leq \E\big( \Phi_\eps(x) \Phi_\eps(y) \big) \leq \frac{\Lambda \eps^{\alpha}}{\big(|x-y|+\eps\big)^{\alpha}}
	\end{equation}
	for some $\Lambda > 1$ and for all $x,y \in \R^d$ and all $\eps \in (0,1)$. Then, for every $K \geq 1$ and $m, r \in \N$, there exists $C>0$ depending on these parameters and $\Lambda$ only such that
	\begin{equation} \label{eq:main}
	\bigg| \E \prod_{j=1}^{K} \d_{\theta}^{r} \hH_{m}\big( e^{i \theta \Phi_\eps(x_j)} \big) \bigg| \leq C \E \prod_{j=1}^{K} \Big( \Phi_{\eps}^{\diamond m}(x_j) + \Phi_{\eps}^{\diamond (m+1)}(x_j) \Big)
	\end{equation}
	for all $\eps \in (0,1)$, $\theta \in \R$ and $\x = (x_k)_{k=1}^{K}$. 
\end{thm}

Theorem~\ref{th:main} is the main technical ingredient to establish that if $\{\Psi_\eps\}$ approximates a certain Gaussian random field $\Psi$, then a large class of nonlinear functions of $\Psi_\eps$, after proper rescaling and re-centering, converges to certain Wick powers of $\Psi$. We first give the assumption on the random field $\Psi$. 

\begin{assumption} \label{as:field}
	$\Psi$ is a stationary Gaussian random field with correlation\footnote{A rigorous way of saying the correlation is $G$ is that
	\begin{equation*}
	\E \scal{\Psi, \varphi} \scal{\Psi, \phi} = \int_{\R^d} G(x-y) \varphi(x) \phi(y) dx dy
	\end{equation*}
    for all $\varphi, \phi \in \cC_{c}^{\infty}(\R^d)$.}
	\begin{equation*}
	\E \big( \Psi(x) \Psi(y) \big) = G(x-y), 
	\end{equation*}
	where $G$ satisfies the bounds
	\begin{equation*}
	\frac{c}{|x|^\alpha} \leq G(x) \leq \frac{C}{|x|^\alpha}\  \quad \text{and} \quad 
	|(\d_j G)(x)| \leq C |x|^{-\alpha-s_j}
	\end{equation*}
	for some $\alpha \in (0,|\fs|)$ and all $x \in \R^d$. In addition, there exists a locally integrable function $g$ such that
	\begin{equation} \label{eq:G_converge}
	\eps^{\alpha} G(\eps^{s_1} x_1, \dots, \eps^{s_d}x_d) \rightarrow g(x)
	\end{equation}
	in $L^{1}(\Omega)$ for every bounded subset $\Omega$ of $\R^d$. 
\end{assumption}

For $M \in \N$ and open subset $\iI \subset \R$, we define the norm $\|\cdot\|_{M,\iI}$ on distributions on $\R$ by
\begin{equation*}
\|\Ups\|_{M,\iI} := \sup_{0 \leq r \leq M} \sup_{\stackrel{\varphi \in \cC_{c}^{M}(\iI):}{\|\varphi\|_{\cC^{M}(\iI)} \leq 1}} |\scal{\Ups, \varphi^{(r)}}|. 
\end{equation*}
Our assumption on the function $F: \R \rightarrow \R$ is the following. 

\begin{assumption} \label{as:F}
	There exists $M \in \N$ such that the Fourier transform of $F$ satisfies
	\begin{equation*}
	\sum_{k \in \Z} \|\hF\|_{M,\iI_k} < +\infty, 
	\end{equation*}
	where $\iI_k = (k-1,k+1)$. 
\end{assumption}

For every $\rho \in \cC_{c}^{\infty}(\R^d)$ and $\eps>0$, let
\begin{equation*}
\rho_{\eps}(x) = \eps^{-|\fs|} \rho(x_{1}/\eps^{s_1}, \dots, x_{d}/\eps^{s_d}). 
\end{equation*}
The main convergence theorem is the following. 

\begin{thm} \label{th:convergence}
	Let $\Psi$ and $F$ satisfy the above assumptions, and $g$ be the limiting $L^1$ function of $\eps^\alpha G(\eps \cdot)$ as in Assumption~\ref{as:field}. Let $\rho$ be a mollifier on $\R^d$ and $\Psi_{\eps} = \Psi * \rho_{\eps}$. For every integer $m$, define
	\begin{equation} \label{eq:a_m}
	a_{m} := \frac{1}{m!} \big( F^{(m)}*\mu \big)(0), 
	\end{equation}
	where $\mu \sim \nN(0,\sigma^2)$ is a Gaussian measure on $\R$ with variance
	\begin{equation} \label{eq:sigma2}
	\sigma^2 = \int g(x-y) \rho(x) \rho(y) dx dy. 
	\end{equation}
	Then for every $m < \frac{|\fs|}{\alpha}$ and every sufficiently small $\kappa$, we have
	\begin{equation*}
	\eps^{-\frac{m \alpha}{2}} \hH_{m}\big(F(\eps^{\frac{\alpha}{2}} \Psi_{\eps})\big) \rightarrow a_{m} \Psi^{\diamond m}
	\end{equation*}
	as $\eps \rightarrow 0$ almost surely in $\cC^{-\frac{m\alpha}{2}-\kappa}$. Here, $\Psi^{\diamond m}$ is the $m$-th Wick power of $\Psi$. 
\end{thm}

We will first prove the main bound \eqref{eq:main} in Theorem~\ref{th:main}, and then establish the convergence in Theorem\ref{th:convergence} by Fourier expanding $F$ and applying \eqref{eq:main} to $\Phi_\eps = \eps^{\frac{\alpha}{2}} \Psi_\eps$. Note that although the bound in Theorem~\ref{th:main} holds for every integer $m$, the convergence in Theorem~\ref{th:convergence} requires $m<\frac{|\fs|}{\alpha}$. This can be easily seen from the fact that if $m \geq \frac{|\fs|}{\alpha}$, then $\Psi^{\diamond m}$ would have divergent covariance and hence not well-defined. 

\begin{eg}
	One typical example of the random field $\Psi$ satisfying Assumption~\ref{as:field} is $\Psi = \lL^{-\frac{\beta}{2}} \xi$, where $\xi$ is the white noise on $\R^d$, $\beta = \frac{1}{2}(|\fs|-\alpha)$, and $\lL$ is the differential operator given by
	\begin{equation*}
	\lL = \sum_{j=1}^{d} (-\d_j^2)^{\frac{1}{s_j}}. 
	\end{equation*}
	This is the fractional Gaussian field. In this case, the convolution kernel $K$ of $\lL^{-\frac{\beta}{2}}$ is homogeneous (in the scaling $\fs$) of order $-|\fs|+\beta$ in the sense that
	\begin{equation*}
	K(\lambda^{s_1}x_1, \dots, \lambda^{s_d}x_d) = \lambda^{-|\fs|+\beta} K(x_1, \dots, x_d)
	\end{equation*}
	for all $\lambda > 0$. Hence, for $g, G$ as in Assumption~\ref{as:field}, we have the expression
	\begin{equation*}
	g(x) = G(x) = \int_{\R^d} K(x+y) K(y) dy. 
	\end{equation*}
	The same is true when $\R^d$ is given the parabolic scaling $(2, 1, \dots, 1)$ and $\lL = \d_t - \Delta$ is the heat operator. 
	
	In the case of standard Euclidean scaling $\fs = (1, \dots, 1)$, $\Psi$ is simply the standard fractional Gaussian field $(-\Delta)^{-\frac{\beta}{2}} \xi$. We refer to the survey \cite{frac_Gaussian} for more details on fractional Gaussian fields. 
\end{eg}

\begin{eg}
As for the function $F$, all $\cC^{1+}$ functions with polynomial growth fall in Assumption~\ref{as:F}. More precisely, if $f \in \cC^{1,\beta}(\R)$ for some $\beta>0$, and there exist $C, M>0$ such that
\begin{equation*}
|f(x)| + |f'(x)| + \sup_{|h|<1} \frac{|f'(x+h)-f'(x)|}{|h|^{\beta}} \leq C (1+|x|)^{M}
\end{equation*}
for all $x \in \R$, then $F$ satisfies Assumption~\ref{as:F}. 
\end{eg}

\begin{eg}
One very interesting example of the microscopic model \eqref{eq:KPZ_micro} is with $F(x)=|x|$. It is almost linear, but one still expects its large scale behaviour to be nonlinear as described by the KPZ equation. This function $F$ even not $\cC^1$, but we still have
\begin{equation*}
\|\hF\|_{1,\iI_k} \leq C (1+|k|)^{-2}, 
\end{equation*}
which clearly satisfies Assumption~\ref{as:F}. Hence, as a consequence of Theorem~\ref{th:main}, if $\Psi = \d_x P * \xi$ where $\xi$ is the space-time white noise with one space dimension, and $P$ is the heat kernel, then we have
\begin{equation*}
\eps^{-\frac{1}{2}} \big( |\Psi_\eps| - \E |\Psi_\eps| \big) \rightarrow a \Psi^{\diamond 2}
\end{equation*}
for some $a>0$ depending on the mollifier. This is the first step towards establishing convergence to the KPZ equation for the microscopic model of the form \eqref{eq:KPZ_micro} with $F(x)=|x|$. 
\end{eg}

\subsection{Remarks and possible generalisations}

Theorem~\ref{th:main} is a special case of \cite[Theorem~6.4]{KPZ_general} in that it allows only one frequency variable $\theta$ rather than multiple ones. On the other hand, it is also more general since it allows subtraction of Wiener chaos up to any order. Furthermore, the bound \eqref{eq:main} is completely independent of $\theta$, while the corresponding one in \cite{KPZ_general} is polynomial in $\theta$. As a consequence of this improvement, the condition on $F$ for the convergence in Theorem~\ref{th:convergence} to hold is weaker. 

The main technical difference that results in this improvement, as we shall see later in Section~\ref{sec:main_bound}, is that in the clustering procedure, we are able to take the clustering distance $L$ being independent of $\theta$ rather than being quadratic in $\theta$ as in \cite{KPZ_general}. 

We shall note that the convergence results in this article are not sufficient to establish weak universality in general situations. These would require convergence of the products of the objects considered in Theorem~\ref{th:convergence}, with possible subtraction of extra chaos components after taking product. The convergence of these products requires a more general bound than Theorem~\ref{th:main} and \cite[Theorem~6.2]{KPZ_general}. We leave them to future work. 

\subsection*{Acknowledgement}

{\small The author acknowledges the support from the Engineering and Physical Sciences Research Council through the fellowship EP/N021568/1. I also thank the anonymous referee for carefully reading the draft version of the article and providing helpful suggestions on improving the presentation.}

\section{Proof of Theorem~\ref{th:main}}
\label{sec:main_bound}

This section is devoted to the proof of Theorem~\ref{th:main}. Assumptions~\ref{as:field} and~\ref{as:F} on $\Psi$ and $F$ are irrelevant here. We fix $\alpha \in (0,|\fs|)$, and let $\{\Phi_{\eps}\}_{\eps \in (0,1)}$ be a family of Gaussian random fields with correlation functions satisfying \eqref{eq:field_eps}. The following preliminary bounds on the correlation function will be used throughout the section. 

\begin{prop} \label{pr:corr_change}
	Let $\gamma \geq 1$. If $|x'-y'| \leq \gamma |x-y|$, then
	\begin{equation} \label{eq:corr_change}
	\E \big( \Phi_{\eps}(x) \Phi_{\eps}(y) \big) \leq \gamma^{\alpha} \Lambda^2 \E \big( \Phi_{\eps}(x') \Phi_{\eps}(y') \big). 
	\end{equation}
	The bound is uniform over all $\eps \in (0,1)$ and all pairs of points $(x,y), (x',y') \in (\R^d)^{2}$ satisfying the above constraint. As a consequence, we have
	\begin{equation} \label{eq:corr_triangle}
	\E\big(\Phi_{\eps}(x)\Phi_{\eps}(y)\big) \; \E\big(\Phi_{\eps}(x)\Phi_{\eps}(z)\big) \leq \frac{2^\alpha \Lambda^3 \eps^{\alpha} }{\big((|x-y| \wedge |x-z|)+\eps\big)^{\alpha}} \cdot \E \big(\Phi_{\eps}(y)\Phi_{\eps}(z)\big)
	\end{equation}
	for all $\eps \in (0,1)$ and all $x,y,z \in \R^d$. 
\end{prop}
\begin{proof}
	The first bound follows from
	\begin{equation*}
	\E \big( \Phi_{\eps}(x') \Phi_{\eps}(y') \big) \geq \frac{\eps^\alpha}{\Lambda \big(|x'-y'|+\eps\big)^{\alpha}} \geq \frac{1}{\Lambda^2 \gamma^\alpha} \cdot \frac{\Lambda \eps^\alpha}{\big(|x-y|+\eps\big)^{\alpha}} \geq \frac{\E \big(\Phi_{\eps}(x)\Phi_{\eps}(y)\big)}{\Lambda^2 \gamma^\alpha}, 
	\end{equation*}
	where we have used the assumption $\gamma \geq 1$ in the second inequality. As for the second one, it suffices to notice
	\begin{equation*}
	|y-z| \leq 2 \max \{|x-y|, |x-z|\}
	\end{equation*}
	and then apply \eqref{eq:corr_change}. 
\end{proof}

In what follows, we keep our notations same as in \cite[Section~6]{KPZ_general}. For every finite set $\aA$, let $\N^{\aA}$ be the set of multi-indices on $\aA$. For $\aA$-tuple of Gaussian random variables $\X = (X_{a})_{a \in \aA}$ and $\n \in \N^{\aA}$, we write $\X^{\diamond \n} = \bdiamond_{a \in \aA} X_{a}^{\diamond n_a}$. Similarly, we write $\n! = \prod_{a \in \aA} n_{a}!$ and $|\n| = \sum_{a \in \aA} n_a$. In general, we use standard letters for scalars, and boldface ones to denote tuples. 

Fix $K \geq 1$ and $m, r \in \N$. Let $[K] = \{1, \dots, K\}$. We also fix $\theta \in \R$ and $\x = (x_k)_{k \in [K]} \in \R^K$ arbitrary. Write $\scal{\theta} = 1 + |\theta|$. All the constants $C$ below depend on $\Lambda$, $K$, $m$ and $r$ only unless otherwise mentioned. We seek bounds that are uniform in $\eps$, $\theta$ and $\x$. We also write $X_j = \Phi_{\eps}(x_j)$ for simplicity since the bounds will be independent of $\eps$.

\subsection{Clustering and the first bound}
\label{sec:clustering}

Let $L>0$ be a fixed large constant whose value, depending on $\Lambda$, $K$, $m$ and $r$ only, will be specified later. Let $\sim$ be an equivalence relation on $[K]$ such that $j \sim j'$ if there exists $k \in \N$ and $j_0, \dots, j_k \in [K]$ with $j_0 = j$ and $j_k = j'$ such that
\begin{equation*}
|x_{j_{\ell+1}} - x_{j_\ell}| \leq L \eps
\end{equation*}
for all $\ell = 0, \dots, k-1$. We let $\Clus$ denote the partition of $[K]$ into clusters obtained in this way. In other words, $j$ and $j'$ belong to the same cluster if and only if starting from $x_{j}$, one can reach $x_{j'}$ by performing jumps with sizes at most $L \eps$ onto connecting points in $\x$. 

We distinguish two cases depending on whether $\Clus$ contains singletons or not. We first prove \eqref{eq:main} when $\Clus$ has no singleton, that is, every cluster in $\Clus$ has at least two elements. In this case, we write down the explicit expression
\begin{equation} \label{eq:chaos_expan_multi}
\d_\theta^r \hH_{m} \big( e^{i \theta X_j} \big) = (i X_j)^{r} e^{i \theta X_j} - \sum_{n \leq m-1} \frac{i^n}{n!} \; \d_{\theta}^{r} \big( e^{-\frac{\theta^2 \E X_j^2}{2}} \theta^n \big) \; X_{j}^{\diamond n}. 
\end{equation}
Since $\E X_j^2 \in [\Lambda^{-1}, \Lambda]$, the coefficients $\d_\theta^r (e^{-\frac{\theta^2 \E X_j^2}{2}}\theta^n)$ are uniformly bounded in $\theta$. We then plug the expansion \eqref{eq:chaos_expan_multi} into the left hand side of \eqref{eq:main}. The Gaussianity of the $X_j$'s and boundedness of the coefficients of the removed chaos imply
\begin{equation} \label{eq:multi_1}
\left| \E \prod_{j=1}^{K} \d_{\theta}^{r} \hH_{m} \big( e^{i \theta X_j} \big) \right| \leq C
\end{equation}
for some $C$ independent of $\eps$, $\theta$ and $\x$. It then remains to show that the right hand side of \eqref{eq:main} is bounded by a constant from below. For this, we use the assumption that $\Clus$ contains no singletons. 

Let $\set \in \Clus$ be arbitrary, and we label its elements by $\set = \{j_1, \dots, j_{|\set|}\}$. Since $\Clus$ has no singleton set, we necessarily have $|\set| \geq 2$. By clustering, any two points in $\set$ are at most $K L \eps$ away from each other. Hence, the assumption \eqref{eq:field_eps} implies (recalling that $X_j = \Phi_\eps(x_j)$)
\begin{equation*}
\E \big(X_j X_{j'} \big) \geq \frac{\eps^\alpha}{\Lambda \big( |x_{j}-x_{j'}| + \eps \big)^{\alpha}} \geq \frac{1}{\Lambda (KL+1)^{\alpha}}
\end{equation*}
for every two points $j, j' \in \set$. which implies
\begin{equation*}
\big( \Lambda (KL+1)^{\alpha} \big)^{-\lfl \frac{m+1}{2} \rfl |\set|} \leq \Big( \prod_{\ell=1}^{|\set|} \E \big( X_{j_\ell} X_{j_{\ell+1}} \big) \Big)^{\lfl \frac{m+1}{2} \rfl}, 
\end{equation*}
where we identified $j_{|\set|+1}$ with $j_1$. Multiplying the above bound over all $\set \in \Clus$ and using Wick's formula and positivity of the correlations, we obtain
\begin{equation} \label{eq:multi_2}
\big( \Lambda (KL+1)^{\alpha} \big)^{-\lfl \frac{m+1}{2} \rfl K} \leq \prod_{\set \in \Clus} \Big( \prod_{\ell=1}^{|\set|} \E \big( X_{j_\ell} X_{j_{\ell+1}} \big) \Big)^{\lfl \frac{m+1}{2} \rfl} \leq \E \prod_{j=1}^{K} \big( X_{j}^{\diamond m} + X_{j}^{\diamond (m+1)} \big). 
\end{equation}
This is the place where we use $|\set| \geq 2$ for every $\set \in \Clus$, for otherwise the middle term above would contain the variance of a single random variable and the second inequality in \eqref{eq:multi_2} would be wrong. Combining \eqref{eq:multi_1} and \eqref{eq:multi_2}, we obtain
\begin{equation} \label{eq:multi}
\left| \E \d_{\theta}^{r} \prod_{j=1}^{K} \hH_{m} \big( e^{i \theta X_j} \big) \right| \leq C \big(\Lambda (KL+1)^{\alpha}\big)^{K \lfl \frac{m+1}{2} \rfl} \cdot \E \prod_{j=1}^{K} \big( X_{j}^{\diamond m} + X_{j}^{\diamond (m+1)} \big), 
\end{equation}
which matches the right hand side of \eqref{eq:main}. Since $L$ (to be chosen later) is also independent of $\eps$, $\theta$ and $\x$, this concludes the case when $\Clus$ contains no singleton. The rest of the section is devoted to establishing \eqref{eq:main} when at least one cluster in $\Clus$ is singleton. 

\subsection{Expansion}

Given the collection of points $\x$ and the clustering above, let
\begin{equation*}
\sS = \big\{ \set \in \Clus: |\set|=1 \big\}
\end{equation*}
be the set of singletons in $\Clus$, and let $\uU = \Clus \setminus \sS$. We write $s \in \sS$ for simplicity if $\{s\}$ is a singleton set in $\Clus$. 

For $\X = (X_j)_{j \in [K]}$ and $\n = (n_j)_{j \in [K]}$, we write $\X_\set$ and $\n_\set$ for their restrictions to $\set$, and $\X_{\set}^{\diamond \n_\set} = \bdiamond_{j \in \set} X_{j}^{\diamond n_j}$. Splitting the left hand side of \eqref{eq:main} into sub-products within clusters $\set \in \Clus$ and chaos expanding each sub-product, we can re-write it as
\begin{equation} \label{eq:expansion_clus}
\E \prod_{j=1}^{K} \d_{\theta}^{r} \hH_{m} \big( e^{i \theta X_j} \big) = \sum_{N \geq 0} \sum_{\stackrel{\n \in \N^K:}{|\n|=N}} \Big( \prod_{\set \in \Clus} C_{\n_\set}(\theta,\X_\set) \Big) \Big( \E \prod_{\set \in \Clus} \X_{\set}^{\diamond \n_\set} \Big). 
\end{equation}
Here, $C_{\n_\set}(\theta, \X_\set)$ is the coefficient of $\X_{\set}^{\diamond \n_\set}$ in the chaos expansion of $\prod_{j \in \set} \d_\theta^r \hH_{m}\big(e^{i\theta X_j}\big)$, and has the expression
\begin{equation} \label{eq:coeff}
C_{\n_\set}(\theta, \X_\set) = \frac{1}{\n_\set!}\; \E \Big( \prod_{j \in \set} \d_\theta^r \d_{X_j}^{n_j} \hH_{m}\big( e^{i\theta X_j} \big) \Big). 
\end{equation}
Note the product involving the expectation on the right hand side of \eqref{eq:expansion_clus} is $0$ if $|\n|$ is odd. But we still sum over all integers $N$ since this simplifies the notations later. The following lemma gives control on the coefficients $C_{\n_\set}$. 

\begin{lem} \label{le:coeff}
	There exists $C>0$ depending on $K$, $m$, $r$ and $\Lambda$ only such that
	\begin{equation} \label{eq:coeff_clus_one}
	|C_{\n_\set}(\theta,\X_\set)| \leq \frac{(C \scal{\theta})^{|\n_\set|}}{\n_\set!}
	\end{equation}
    for all $\n_\set \in \N^{\set}$, where we recall $\scal{\theta} = 1 + |\theta|$. As a consequence, we have
	\begin{equation} \label{eq:coeff_clus}
	\sum_{|\n|=N} \prod_{\set \in \Clus} |C_{\n_\set}(\theta,\X_\set)| \leq \frac{(C \scal{\theta})^{N}}{N!}. 
	\end{equation}
	Furthermore, if $\sS \neq \emptyset$, then we have
	\begin{equation} \label{eq:coeff_singleton}
	\sum_{|\n|=N} \prod_{\set \in \Clus} |C_{\n_\set}(\theta,X_\set)| \leq e^{-\frac{\theta^2}{2\Lambda}} \cdot \frac{(C \scal{\theta})^{N}}{N!}. 
	\end{equation}
	All the bounds are uniform in $\eps$ and $\theta$, and in the location of $\x$ subject to whether $\Clus$ contains a singleton or not. 
\end{lem}
\begin{proof}
	We express the right hand side of \eqref{eq:coeff} in a way that is convenient to estimate. For this, we introduce variables $\Bbeta \in \R^{K}$. Let $\Bbeta_\set$ denote the restriction of $\Bbeta$ to $\set$, and write $\d_{\Bbeta_\set}^{r} = \prod_{j \in \set}\d_{\beta_j}^{r}$ and $\Bbeta_\set^{\n_\set} = \prod_{j \in \set} \beta_{j}^{n_j}$. Using the identity
	\begin{equation*}
	\d_{X_j}^{n_j} \hH_{m}\big( e^{i \beta_j X_j} \big) = (i\beta_{j})^{n_j} \hH_{m-n_j}\big( e^{i \beta_j X_j} \big)
	\end{equation*}
	where $\hH_{m-n} = \hH_0$ if $n \geq m$, we can re-write the right hand side of \eqref{eq:coeff} as
	\begin{equation*}
	C_{\n_\set}(\theta, \X_\set) = \frac{i^{|\n_\set|}}{\n_\set!} \d_{\Bbeta_\set}^{r} \Big( \Bbeta_\set^{\n_\set} \cdot  \E \prod_{j \in \set} \hH_{m-n_j}\big( e^{i \beta_j X_j} \big) \Big) \Big|_{\beta_j=\theta\;, \forall j \in \set}. 
	\end{equation*}
	When distributing $r$ derivatives of each $\beta_j$ into the two terms in the parenthesis, the differentiation of the first term ($\Bbeta_\set^{\n_\set}$) yields an additional factor which is at most $|\n_\set|^{Kr} \leq C^{|\n_\set|}$ for some fixed constant $C$, while the second term is uniformly bounded both in $\theta$ and $\n_\set$ since $X_j$'s all have bounded variance. This gives \eqref{eq:coeff_clus_one}. The bound \eqref{eq:coeff_clus} follows from \eqref{eq:coeff_clus_one} and the multinomial theorem. 
	
	Finally, in order to obtain \eqref{eq:coeff_singleton} when $\sS \neq \emptyset$, it suffices to note that for $s \in \sS$, we have
	\begin{equation*}
	C_{n_s}(\theta,X_s) = \frac{i^{n_s}}{n_{s}!} \d_{\theta}^{r} \big( \theta^{n_s} e^{-\frac{\theta^2 \E X_s^2}{2}} \big)
	\end{equation*}
	if $n_s \geq m$, and is $0$ otherwise. Since $\E X_s^2 \geq \frac{1}{\Lambda}$, we gain an additional Gaussian factor $e^{-\frac{\theta^2}{2 \Lambda}}$ for every $s \in \sS$, and hence $e^{-\frac{\theta^2 |\sS|}{2 \Lambda}}$ in total. The bound \eqref{eq:coeff_singleton} then follows from relaxing it to $e^{-\frac{\theta^2}{2 \Lambda}}$. 
\end{proof}

\subsection{Representative point}

For $|\set| \geq 2$, the corresponding term in the expectation on the right hand side of \eqref{eq:expansion_clus} is a Wick product of multiple Gaussian random variables. We aim to reduce it to the Wick product of a single variable by choosing a representative point from each cluster. 

For every $\set \in \Clus$, choose $u^{*}(\set) \in \set$ arbitrary. The choice for $u^*$ is unique if $\set$ is singleton. We have the following proposition. 

\begin{prop} \label{pr:representative}
	There exists $C>0$ such that
	\begin{equation} \label{eq:representative}
	\E \Big(\prod_{\set \in \Clus} \X_{\set}^{\diamond \n_\set} \Big) \leq C^{|\n|} \cdot \E \Big( \prod_{\set \in \Clus} X_{u^*(\set)}^{\diamond |\n_\set|} \Big)
	\end{equation}
    for every $\n \in \N^K$ and every choice of $u^*(\set) \in \set$. 
\end{prop}
\begin{proof}
	If $|\n| = \sum_{\set \in \Clus} |\n_\set|$ is odd, then both sides of \eqref{eq:representative} are $0$, so we only need to consider the situation when $|\n|$ is even. 
	
	In this case, the left hand side is the sum of products of pairwise expectations $\E(X_j X_{j'})$ for $j$ and $j'$ belonging to different clusters. The right hand side (without the factor $C^{|\n|}$) is the same except that each instance of $X_j$ for $j \in \set$ is replaced by $X_{u^*(\set)}$. It then suffices to control the effects of such replacements. 
	
	Let $\set, \fv$ be two different clusters in $\Clus$, and use $u^*$ and $v^*$ to denote $u^*(\set)$ and $u^*(\fv)$ respectively. For every $j \in \set$ and $j' \in \fv$, according to the clustering, we have
	\begin{equation*}
	|x_{u^*} - x_{j}| \leq (K-1) |x_{j}-x_{j'}|\;, \quad |x_{v^*} - x_{j'}| \leq (K-1) |x_{j} - x_{j'}|\;, 
	\end{equation*}
	which implies
	\begin{equation*}
	|x_{u^*} - x_{v^*}| \leq |x_{u^*} - x_j| + |x_j - x_{j'}| + |x_{v^*} - x_{j'}| \leq 2K |x_{j} - x_{j'}|. 
	\end{equation*}
	Hence, by \eqref{eq:corr_change}, we deduce that
	\begin{equation*}
	\E(X_{j} X_{j'}) \leq (2K)^{\alpha} \Lambda^2 \; \E(X_{u^*} X_{v^*}).   
	\end{equation*}
	This is the effect of one such replacement. The claim then follows since there are $\frac{|\n|}{2}$ pairwise expectations in each product in the sum. It also means we can take $C = (2K)^{\frac{\alpha}{2}} \Lambda$ in \eqref{eq:representative}. 
\end{proof}

From now on, we restrict ourselves to the situation when $|\sS| \geq 1$, and recall the notation $\uU = \Clus \setminus \sS$. We need to split the product on the right hand side of \eqref{eq:expansion_clus} into sub-products in $\sS$ and in $\uU$. For this, we introduce multi-indices $\k=(k_s)_{s \in \sS}$ and $\Bell = (\ell_\set)_{\set \in \uU}$, and write $|\k| = \sum_{s \in \sS} k_s$ and $|\Bell| = \sum_{\set \in \uU} \ell_\set$. We then have the following proposition. 

\begin{prop} \label{pr:relax}
	Suppose $|\sS| \geq 1$. Then the left hand side of \eqref{eq:main} can be controlled by
    \begin{equation} \label{eq:main_relax}
    \begin{split}
    \bigg| \E \prod_{j=1}^{K} \d_{\theta}^{r} \hH_{m} \big( e^{i \theta X_j} \big) \bigg| &\leq C e^{-\frac{\theta^2}{2 \Lambda}} \sum_{N \geq 0} \bigg( \frac{\big( C \scal{\theta} \big)^{N+m|\sS|}}{(N+m|\sS|)!} \times\\
    &\sup_{|\k|+|\Bell|=N} \E \Big[ \Big(\prod_{s \in \sS} X_{s}^{\diamond (m+k_s)}\Big) \Big(\prod_{\set \in \uU} X_{u^*(\set)}^{\diamond \ell_\set}\Big) \Big] \bigg). 
    \end{split}
    \end{equation}
\end{prop}
\begin{proof}
	We start with the expression \eqref{eq:expansion_clus}. Note that for $s \in \sS$, $C_{n_s}(\theta, X_s)=0$ whenever $n_s < m$, so we can relax the expression to
	\begin{equation*}
	\bigg| \E \prod_{j=1}^{K} \d_{\theta}^{r} \hH_{m} \big( e^{i \theta X_j} \big) \bigg| \leq \sum_{N \geq m|\sS|} \Big( \sum_{\n} \prod_{\set \in \Clus} |C_{\n_\set}(\theta,\X_{\n_\set})| \Big) \Big( \sup_{\n} \E \prod_{\set \in \Clus} \X_{\set}^{\diamond \n_\set} \Big), 
	\end{equation*}
	where both the sum and supremum are taken over $|\n|=N$ with the further restriction that $n_s \geq m$ for all $s \in \sS$. The claim follows immediately by applying Lemma~\ref{le:coeff} and Proposition~\ref{pr:representative} to the right hand side above and noting the range of the sum and supremum. 
\end{proof}

\subsection{Graphic representation}

It remains to control the term involving the expectation on the right hand side of \eqref{eq:main_relax}. Since all $X_j$'s are Gaussian, it can be written as a sum over products of pairwise expectations. The number of terms in each product (and hence the total power) can be arbitrarily large since $N$ will be summed over all integers. Following \cite{KPZ_general}, we introduce graphic notations to describe these objects. 

Given a set $\VV$, we write $\VV_2$ for the set of all subsets of $\VV$ with exactly two elements. A (generalised) graph is a triple $\Gamma = (\VV, \EE, R)$. Here, $\VV$ is the set of vertices, and $\EE: \VV_2 \rightarrow \N$ is the set of edges with multiplicities. More precisely, each edge $\{x,y\} \in \VV_2$ has multiplicity $\EE(x,y) = \EE(y,x)$. We do not allow self-loops, so $\EE(x,x)=0$ for all $x \in \VV$. Finally, $R: \VV_2 \rightarrow \R$ is a function that assigns a value to each pair of vertices. 

Given a graph $\Gamma = (\VV, \EE, R)$, we define the degree of a point $x \in \VV$ and of $\Gamma$ respectively by
\begin{equation*}
\deg(x) := \sum_{y \in \VV} \EE(x,y)\;, \qquad \deg(\Gamma) := \sum_{x \in \VV} \deg(x). 
\end{equation*}
The value of $\Gamma$ is defined by
\begin{equation*}
|\Gamma| := \prod_{e \in \VV_2} \big(R(e)\big)^{\EE(e)}. 
\end{equation*}
In what follows, we always take $\VV = \{x_j\}_{j=1}^{K}$ fixed (so is the clusters in $\Clus$), and $R(x_j, x_{j'}) = \E(X_j X_{j'})$. We also fix the representative points $u^*(\set)$ chosen for each $\set \in \Clus$. Hence, the only variable in our graph is the multiplicity $\EE$ of the edges. Recall the decomposition $\Clus = \sS \cup \uU$ into singletons and clusters with at least two points. We introduce the following definition to characterise the pairings that appear in the expectation term on the right hand side of \eqref{eq:main_relax}. 

\begin{defn} \label{de:graph}
	For each $\k \in \N^{\sS}$ and $\Bell \in \N^{\uU}$, the set $\Omega_{\k,\Bell}$ consists of graphs with $\VV$ and $R$ specified above, and such that $\deg(x_s)=m+k_s$ for all $s \in \sS$, $\deg(x_{u^*(\set)})=\ell_\set$ for all $\set \in \uU$, and $\deg(x)=0$ for all other $x \in \VV$. 
	
	Let $\Omega^*$ be the set of graphs $\Gamma$ such that both of the following hold: 
	\begin{enumerate}
		\item $\Gamma \in \Omega_{\k,\Bell}$ for some $\k \in \N^{\sS}$ and $\Bell \in \N^{\uU}$ with the restriction that $k_s \in \{0,1\}$ for every $s \in \sS$ and $\ell_\set \leq m+1$ for each $\set \in \uU$. 
		
		\item If $\ell_\set \geq 1$ for some $\set \in \uU$, then there exists $s \in \sS$ such that $\EE(x_s, x_{u^*(\set)}) = \ell_\set$. 
	\end{enumerate}
\end{defn} 

\begin{rmk}
	The first requirement for $\Omega^*$ above is equivalent to that $\deg(x_s) \in \{m,m+1\}$ for all $s \in \sS$, $\deg(x_{u^*(\set)}) \leq m+1$ for all $\set \in \uU$, and is $0$ for all other points. The second requirement says that if $x_{u^*(\set)}$ has a non-zero degree, then all its edges must be connected to a single point $x_s$ for some $s \in \sS$. We will see later that the definition of $\Omega^*$ corresponds to ``minimal graphs" after the reduction procedure in the next subsection. 
\end{rmk}

\begin{rmk}
	The clustering depends on the choice of $L$, and hence so do the definitions of $\Omega_{\k,\Bell}$ and $\Omega^*$. On the other hand, these are just intermediate steps and our final bound \eqref{eq:main} does not involve clustering at all. Furthermore, the choice of $L$ later (in \eqref{eq:choice_L}) is also independent of the location of $\x$. Hence, we omit the dependence of the clustering on $L$ here for notational simplicity. 
\end{rmk}

\subsection{Reduction}

We now start to control the right hand side of \eqref{eq:main_relax}. If $m|\sS| + |\k| + |\Bell|$ is odd, then the term with the expectation is $0$. So we only need to deal with the case when $m|\sS| + |\k| + |\Bell|$ is even. 

In that case, the number of different pairings contributing to the expectation in \eqref{eq:main_relax} is at most $(m|\sS|+|\k|+|\Bell|-1)!!$, so with Definition~\ref{de:graph}, we have
\begin{equation} \label{eq:Wick_bound}
\E \Big[ \Big(\prod_{s \in \sS} X_{s}^{\diamond (m+k_s)}\Big) \Big(\prod_{\set \in \uU} X_{u^*(\set)}^{\diamond \ell_\set}\Big) \Big] \leq (|\k|+|\Bell|+m|\sS|-1)!! \cdot \sup_{\Gamma \in \Omega_{\k,\Bell}} |\Gamma|. 
\end{equation}
Comparing the above bound and the right hand side of \eqref{eq:main_relax}, we see that we need to control $|\Gamma|$ for $\Gamma \in \Omega_{\k,\Bell}$ with arbitrarily large $\k$ and $\Bell$. We first bound it by values of the graphs in $\Omega^*$, which is done via a reduction procedure. After that, we enhance the graphs in $\Omega^*$ to match the right hand side of \eqref{eq:main} to conclude the proof. 

We start with the reduction step. This is where we need to choose the clustering distance $L$ sufficiently large, which will ensure the uniform in $\theta$ bound after summing over $\k$ and $\Bell$. We first give the following proposition, which reduces graphs in $\Omega_{\k,\Bell}$ to those in $\Omega^*$. 

\begin{prop} \label{pr:reduction}
	There exists $C>0$ depending on $\Lambda$ only such that
	\begin{equation*}
	\max_{\Gamma \in \Omega_{\k,\Bell}} |\Gamma| \leq  \max_{\Gamma^* \in \Omega^*} \Big( \big(C L^{-\alpha}\big)^{\frac{1}{2} (\deg(\Gamma)-\deg(\Gamma^*))} \cdot |\Gamma^*| \Big)
	\end{equation*}
	for every pair $(\k,\Bell)$ and every $L>0$. The constant $C$ does not depend on the choice of $L$, though the clusters and the definitions of $\Omega_{\k,\Bell}$ and $\Omega^*$ do. 
\end{prop}
\begin{proof}
	Fix $\k \in \N^{\sS}$, $\Bell \in \N^{\uU}$, and $\Gamma \in \Omega_{\k,\Bell}$ arbitrary. It suffices to show that if $\Gamma \notin \Omega^{*}$, then we can find a $\bar{\Gamma} \in \Omega_{\bar{\k}, \bar{\Bell}}$ with $\bar{\k} \leq \k$, $\bar{\Bell} \leq \Bell$ and $|\bar{\k}|+|\bar{\Bell}| < |\k|+|\Bell|$ strictly such that
	\begin{equation} \label{eq:reduction}
	|\Gamma| \leq \big(C L^{-\alpha}\big)^{\frac{1}{2}(|\k-\bar{\k}|+|\Bell-\bar{\Bell}|)} |\bar{\Gamma}|. 
	\end{equation}
	One can then iterate this bound until the graph is reduced to some $\Gamma^* \in \Omega^*$ to conclude the proposition. This necessarily happens since each time the total degree of the graph decreases strictly. Here, the inequality on $\k$ and $\Bell$ means the inequality in each component. 
	
	To see the existence of such a $\bar{\Gamma}$ when $\Gamma \in \Omega_{\k,\Bell} \setminus \Omega^*$, we consider the two situations where either one of the two conditions for $\Omega^*$ in Definition~\ref{de:graph} is violated. We first consider the violation of Condition 1. Since $\Gamma \in \Omega_{\k,\Bell}$, failure of Condition 1 means there exists $j \in \sS \cup \{u^*(\set): \set \in \uU\}$ such that $\deg(x_j) \geq m+2$. We fix this $j$, and there are two possibilities in this situation. 
	
	\begin{flushleft}
		\textit{Case 1.}
	\end{flushleft}
	There exist $i \neq i'$ such that $\EE(x_j, x_i) \geq 1$ and $\EE(x_j, x_{i'}) \geq 1$. In this case, we let $\bar{\Gamma}$ be the graph obtained from $\Gamma$ by performing the following operations: 
	\begin{equation*}
	\EE(x_j, x_{i}) \mapsto \EE(x_{j}, x_{i})-1,\; \EE(x_j, x_{i'}) \mapsto \EE(x_{j}, x_{i'})-1, \EE(x_{i},\; x_{i'}) \mapsto \EE(x_{i}, x_{i'})+1. 
	\end{equation*}
	The only point whose degree has been changed in this operation is $x_j$ (reduced by $2$). Hence, we have $\bar{\Gamma} \in \Omega_{\bar{\k}, \bar{\Bell}}$ with $(\bar{\k},\bar{\Bell}) \leq (\k,\Bell)$ and $|\k-\bar{\k}| + |\Bell-\bar{\Bell}|=2$. To see the bound \eqref{eq:reduction}, we note that by definition of $\Omega_{\k,\Bell}$, $x_j$ is at least $L \eps$ away from both $x_i$ and $x_{i'}$. Hence, by \eqref{eq:corr_triangle}, we have
	\begin{equation*}
	\EE(X_j X_i) \cdot \EE(X_j X_{i'}) \leq \frac{2^\alpha \Lambda^3}{(L+1)^{\alpha}} \cdot \EE(X_i X_{i'}). 
	\end{equation*}
	In graphic notation, this means
	\begin{equation*}\left|
		\begin{tikzpicture}[scale=1,baseline=-0.8cm]
		\node at (-0.8,0) [dot] (left){};
		\node at (0.8,0) [dot] (right) {};
		\node at (0,-1) [dot] (below) {}; 
		\node at (0,-1.3) {\scriptsize $j$}; 
		\node at (-1,0.2) {\scriptsize $i$}; 
		\node at (1,0.2) {\scriptsize $i'$}; 
		\draw (left) to node[labl]{\tiny $a+1$} (below); 
		\draw (right) to node[labl]{\tiny $b+1$} (below); 
		\draw (left) to node[labl]{\tiny $c$} (right); 
		\end{tikzpicture}\right|
		\leq \frac{2^\alpha \Lambda^3}{(L+1)^{\alpha}} \cdot
		\left|
		\begin{tikzpicture}[scale=1,baseline=-0.8cm]
		\node at (-0.8,0) [dot] (left){};
		\node at (0.8,0) [dot] (right) {};
		\node at (0,-1) [dot] (below) {}; 
		\node at (0,-1.3) {\scriptsize $j$}; 
		\node at (-1,0.2) {\scriptsize $i$}; 
		\node at (1,0.2) {\scriptsize $i'$}; 
		\draw (left) to node[labl]{\tiny $a$} (below); 
		\draw (right) to node[labl]{\tiny $b$} (below); 
		\draw (left) to node[labl]{\tiny $c+1$} (right); 
		\end{tikzpicture}\right|\;, 
	\end{equation*}
	where we have omitted $x$ and simply write the indices to denote vertices. Since all the other parts of the graph remain unchanged, this operation gives a desired $\bar{\Gamma}$ with \eqref{eq:reduction}. 
	
	\begin{flushleft}
		\textit{Case 2. }
	\end{flushleft}
    If for the $x_j$ that violates Condition 1, all its edges are connected to another point $x_i$, then we necessarily have $\deg(x_i) \geq m+2$. Thus, we let $\bar{\Gamma}$ be the graph obtained from $\Gamma$ by reducing $\EE(x_j, x_i)$ by two. Then, $\bar{\Gamma} \in \Omega_{\bar{\k},\bar{\Bell}}$ with $(\bar{\k}, \bar{\Bell}) \leq (\k, \Bell)$ but this time $|\k-\bar{\k}|+|\Bell-\bar{\Bell}|=4$. Since $|x_j - x_i| \geq L \eps$, we also have the bound
    \begin{equation*}
    |\Gamma| \leq \frac{\Lambda^2}{(L+1)^{2\alpha}} |\bar{\Gamma}|, 
    \end{equation*}
    which is also of the form \eqref{eq:reduction}. This completes the treatment of the violation of Condition 1. 
    
    \medskip
    
    We now turn to the situation when Condition 2 is violated. This means there exists $\set \in \uU$ such that
    \begin{enumerate} [label=(\alph*)]
    	\item either $x_{u^*(\set)}$ is connected to two other different points $x_i$ and $x_{i'}$; 
    	\item or $x_{u^*(\set)}$ is connected to $x_{u^*(\set')}$ for some $\set' \in \uU$. 
    \end{enumerate}
    For (a), we perform exactly the same operation as Case 1 in the above situation. This will give rise to a graph in $\Omega_{\bar{\k},\bar{\Bell}}$ with $\bar{\k}=\k$, $\bar{\ell}_\set = \ell_\set-2$ and $\bar{\ell}_{\set'}=\ell_{\set'}$ for all other $\set' \in \uU$, and satisfying \eqref{eq:reduction}. For (b), we simply reduce $\EE(x_{u^*(\set)} x_{u^*(\set')})$ by $1$, which results a graph in $\Omega_{\bar{\k}, \bar{\Bell}}$ with $|\k-\bar{\k}|+|\Bell-\bar{\Bell}|=2$ and the desired bound \eqref{eq:reduction}. 
    
    \bigskip
    
    Since the above cases have covered all the possibilities for $\Gamma \in \Omega_{\k,\Bell} \setminus \Omega^*$, we have completed the proof of the proposition. 
\end{proof}

The following proposition is then a simple consequence. 

\begin{prop} \label{pr:simplified}
	There exist $C>0$ and $L>0$ depending on $\Lambda$, $K$, $m$ and $r$ only such that for every $\theta \in \R$, every location $\x \in (\R^d)^K$ and every $\eps \in (0,1)$, we have the bound
	\begin{equation} \label{eq:simplified}
	\bigg| \E \prod_{j=1}^{K} \d_{\theta}^{r} \hH_{m}\big(e^{i \theta X_j}\big) \bigg| \leq C \max_{\Gamma \in \Omega^*} |\Gamma|
	\end{equation}
\end{prop}

\begin{rmk}
	The bound is completely independent of $\theta$ and $\eps$, and its dependence on the location of $\x$ is via $\Omega^*$ only. Also note that the clustering, and hence $\Omega^*$, depend on the choice of $L$. 
\end{rmk}

\begin{proof} [Proof of Proposition~\ref{pr:simplified}]
	Note that graphs in $\Omega_{\k,\Bell}$ have degree $|\k|+|\Bell| + m|\sS|$, so by Proposition~\ref{pr:reduction}, there exists $\Gamma^* \in \Omega^*$ such that
	\begin{equation*}
	\max_{\Gamma \in \Omega_{\k,\Bell}} |\Gamma| \leq (CL^{\frac{\alpha}{2}})^{\deg(\Gamma^*)} \cdot (CL^{-\alpha})^{\frac{1}{2}(|\k|+|\Bell|+m|\sS|)} \cdot |\Gamma^*|
	\end{equation*}
	for all $\k$ and $\Bell$. Combining it with Proposition~\ref{pr:relax} and \eqref{eq:Wick_bound}, we get
	\begin{equation} \label{eq:simplified_intermediate}
	\bigg| \E \prod_{j=1}^{K} \d_{\theta}^{r} \hH_{m}\big(e^{i \theta X_j}\big) \bigg| \leq C L^{\frac{\alpha}{2} \cdot \deg(\Gamma^*)} \cdot \exp \Big(-\frac{\theta^2}{2 \Lambda} + \frac{C_0 \scal{\theta}^{2}}{L^{\alpha}}\Big) \cdot \max_{\Gamma^* \in \Omega^*} |\Gamma^*|
	\end{equation}
	for some constant $C_0$. One can choose $L$ sufficiently large depending on $C_0$ and $\Lambda$ only such that
	\begin{equation} \label{eq:choice_L}
	\frac{1}{L^{\alpha}} < \frac{1}{4 C_0 \Lambda}. 
	\end{equation}
	This guarantees that the exponential term is uniformly bounded in $\theta$. Since $C_0$ depends on $\Lambda$, $K$, $m$ and $r$ only, so does $L$. Finally, $L^{\frac{\alpha}{2} \cdot \deg(\Gamma^*)}$ is also uniformly bounded since graphs in $\Omega^*$ have degrees at most $(m+1)K$. This completes the proof. 
\end{proof}

\begin{rmk} \label{rm:singleton_separate}
	The reason why we need to choose $L$ large is to ensure the exponential and hence the whole right hand side of \eqref{eq:simplified_intermediate} being uniformly bounded in $\theta$. As we see now, the Gaussian factor $e^{-\frac{\theta^2}{2 \Lambda}}$ in \eqref{eq:coeff_singleton} allows us to choose such $L$ being independent of $\theta$. Together with the enhancement procedure in Section~\ref{sec:enhancement} below, this ensures the bounds in Proposition~\ref{pr:simplified} and hence in Theorem~\ref{th:main} are completely independent of $\theta$. 
	
	Without the Gaussian factor, one would need to take $L$ quadratic in $\theta$ to make the exponential in \eqref{eq:simplified_intermediate} bounded, and the enhancement procedure in below would produce a bound that is polynomial in $\theta$ with its degree depending on $m$ and $K$. 
\end{rmk}

\subsection{Enhancement and conclusion of the proof}
\label{sec:enhancement}

From now on, we fix the choice of $L$ in \eqref{eq:choice_L}. We need to control the right hand side of \eqref{eq:simplified} by that of \eqref{eq:main}. To achieve this, we enhance every $\Gamma \in \Omega^*$ to a graph $\Enh(\Gamma)$ where $\deg(x_j) \in \{m,m+1\}$ for every $j \in [K]$, which matches the pairing occurring in the desired upper bound. The enhancement procedure will also be performed in such a way that $|\Enh(\Gamma)|$ is an upper bound for $|\Gamma|$ up to some proportionality constant, which is uniform in $\eps$, $\theta$ and $\x$ subject to $|\sS| \geq 1$. This will lead to the bound \eqref{eq:main}. The procedure is similar to the one used to obtain \eqref{eq:multi} when $|\sS|=0$. 

Fix $\Gamma \in \Omega^*$ arbitrary, so in particular, $\Gamma \in \Omega_{\k,\Bell}$ for some $\k \in \N^\sS$ and $\Bell \in \N^\uU$. By the definition of $\Omega^*$, $\deg(x_s) \in \{m,m+1\}$ for all $s \in \sS$. For every $\set \in \uU$, we have $\deg(x_{u^*})=\ell_\set \leq m+1$, and all of them are connected to one single $x_s$ for some $s \in \sS$ if $\ell_\set \geq 1$. All other points in $\set$ have degree $0$. To construct $\Enh(\Gamma)$, we add new edges to vertices in $\set \in \uU$ and also move around existing edges, but keep $\deg(x_s)$ unchanged for all $s \in \sS$ throughout the procedure. We do this cluster by cluster, and write $u^*=u^*(\set)$ for simplicity. 

Fix $\set \in \uU$ arbitrary. To perform the enhancement operation for $\set$, we let $s \in \sS$ be such that $x_s$ is the unique singleton point connected to $x_{u^*(\set)}$ if $\ell_\set \geq 1$. This also includes $\ell_\set=0$, in which case $s$ could be arbitrary. We distinguish several situations depending on the number of points in $\set$. 

\begin{flushleft}
	\textit{Case 1.} $|\set| = 2$. 
\end{flushleft}

Let $j \neq u^*(\set)$ denote the other point in $\set$. By definition of $\Omega^*$, we have $\deg(x_j)=0$. We then perform the following operations. We move $\lfl (\ell_\set+1)/2 \rfl$ of the $\ell_\set$ edges between $x_s$ and $x_{u^*}$ to connecting $x_s$ and $x_j$, and add $m-\ell_\set$ edges between $x_{u^*}$ and $x_{j}$. By clustering, we have $|x_{u^*}-x_{j}| \leq L \eps$ and $|x_{s}-x_j| \leq 2|x_s-x_{u^*}|$. Hence, Proposition~\ref{pr:corr_change} gives the bounds
\begin{equation*}
\begin{split}
(\E X_s X_{u*})^{\ell_\set} &\leq C (\E X_s X_{u^*})^{\lfl \frac{\ell_\set}{2} \rfl} (\E X_{s} X_j)^{\lfl \frac{\ell_\set+1}{2} \rfl}, \\
1 &\leq C \big( (L+1)^{\alpha} \; \E (X_{u^*}X_{j})\big)^{m-\lfl \frac{\ell_\set}{2} \rfl},
\end{split} 
\end{equation*}
where $L$ as chosen in \eqref{eq:choice_L} is independent of $\theta$, $\eps$ and $\x$. So in graphic notation, the above operation gives
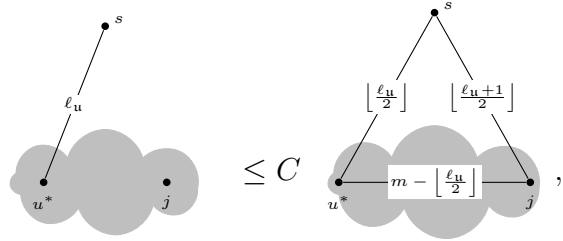
\begin{equation*}
\begin{tikzpicture} [scale=0.9,baseline=1]
\node[cloud, cloud puffs=7.7, cloud ignores aspect, minimum width=2.5cm, minimum height=1.4cm,  draw=lightgray, fill=lightgray]  at (0,0) {};
\node at (0,2.3) [edot] (up) {}; 
\node at (-0.9,0) [edot] (left) {}; 
\node at (0.9,0) [edot] (right) {}; 
\node at (0.2, 2.4) {\tiny $s$}; 
\node at (-0.9,-0.3) {\tiny $u^*$};
\node at (0.9,-0.3) {\tiny $j$};
\draw (up) to node[labl]{\tiny $\ell_\set$} (left); 
\end{tikzpicture}
\quad \leq C \phantom{1}
\begin{tikzpicture} [scale=0.9,baseline=1]
\node[cloud, cloud puffs=7.7, cloud ignores aspect, minimum width=2.8cm, minimum height=1.5cm,  draw=lightgray, fill=lightgray]  at (0,0) {};
\node at (0,2.5) [edot] (up) {}; 
\node at (-1.4,0) [edot] (left) {}; 
\node at (1.4,0) [edot] (right) {};
\node at (0.2, 2.6) {\tiny $s$};  
\node at (-1.4,-0.3) {\tiny $u^*$};
\node at (1.4,-0.3) {\tiny $j$};
\draw (up) to node[labl]{\tiny $\lfl \frac{\ell_\set}{2} \rfl$} (left); 
\draw (up) to node[labl]{\tiny $\lfl \frac{\ell_\set+1}{2} \rfl$} (right); 
\draw (left) to node[labl]{\tiny $m-\lfl \frac{\ell_\set}{2} \rfl$} (right); 
\end{tikzpicture}\;, 
\end{equation*}
where the grey area indicates the cluster $\set$, and we have omitted drawing the remaining $(m-\ell_\set)$ or $(m+1-\ell_\set)$ edges from $x_s$. We also drop $|\cdot|$ and simply use the graph itself to denote its value. Then, $\deg(x_s)=m$ or $m+1$ is unchanged in the procedure. Furthermore, we have $\deg(x_{u^*})=m$ and $\deg(x_j)\in \{m,m+1\}$ after the operation. This also includes the situation $\ell_\set=0$.

\begin{flushleft}
	\textit{Case 2.} $|\set|=3$. 
\end{flushleft}

Let $i, j$ denote the two other points in $\set$. We then perform the operation
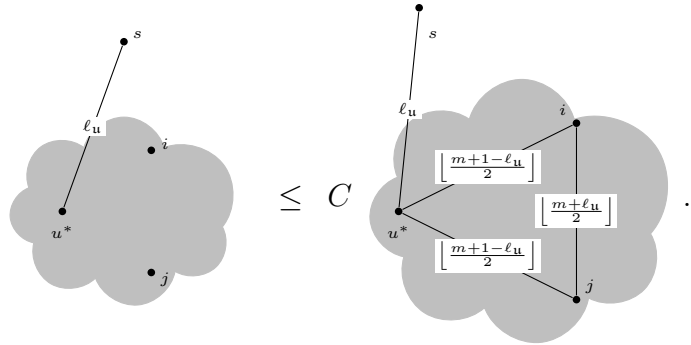
\begin{equation*}
\begin{tikzpicture} [scale=0.9,baseline=2]
\node[cloud, cloud puffs=7.7, cloud ignores aspect, minimum width=3cm, minimum height=2.5cm,  draw=lightgray, fill=lightgray]  at (0,0) {};
\node at (0,2.5) [edot] (up) {}; 
\node at (-0.9,0) [edot] (left) {}; 
\node at (0.4,0.9) [edot] (upright) {}; 
\node at (0.4,-0.9) [edot] (bottomright) {}; 
\node at (0.2, 2.6) {\tiny $s$}; 
\node at (-0.9,-0.3) {\tiny $u^*$};
\node at (0.6,1) {\tiny $i$}; 
\node at (0.6,-1) {\tiny $j$};
\draw (up) to node[labl]{\tiny $\ell_\set$} (left); 
\end{tikzpicture}
\quad \leq \phantom{1} C \phantom{1}
\begin{tikzpicture} [scale=0.9,baseline=2]
\node[cloud, cloud puffs=7.7, cloud ignores aspect, minimum width=4cm, minimum height=3.5cm,  draw=lightgray, fill=lightgray]  at (0,0) {};
\node at (-1.5,3) [edot] (up) {}; 
\node at (-1.8,0) [edot] (left) {}; 
\node at (0.8,1.3) [edot] (upright) {}; 
\node at (0.8,-1.3) [edot] (bottomright) {}; 
\node at (-1.3, 2.6) {\tiny $s$}; 
\node at (-1.8,-0.3) {\tiny $u^*$};
\node at (0.6,1.5) {\tiny $i$}; 
\node at (1,-1.1) {\tiny $j$};
\draw (up) to node[labl]{\tiny $\ell_\set$} (left); 
\draw (left) to node[labl]{\tiny $\lfl \frac{m+1-\ell_\set}{2} \rfl$} (upright); 
\draw (left) to node[labl]{\tiny  $\lfl \frac{m+1-\ell_\set}{2} \rfl$} (bottomright); 
\draw (upright) to node[labl]{\tiny $\lfl \frac{m+\ell_\set}{2} \rfl$} (bottomright); 
\end{tikzpicture}\;. 
\end{equation*}
We see $\deg(x_s)$ is unchanged. One can also check that $\deg(x_{u^*(\set)}) = m$ or $m+1$, and $\deg(x_i) = \deg(x_j) = m$. So we have the correct degrees of the vertices as well as the desired bound. 

\begin{flushleft}
	\textit{Case 3.} $|\set| \geq 4$. 
\end{flushleft}

We denote the other $|\set|-1$ points in the cluster by $j_{1}, \dots, j_{|\set|-1}$. For $|\set|-2$ of them, say $x_{j_1}, \dots, x_{j_{|\set|-2}}$, we perform the same operation as in Section~\ref{sec:clustering} by cyclically connecting them with edges of multiplicities $\lfl \frac{m+1}{2} \rfl$. This yields the bound
\begin{equation*}
1 \leq C \Big( \E(X_{j_1} X_{j_2}) \cdots \E(X_{j_{|\set|-3}} X_{j_{|\set|-2}}) \E(X_{j_{|\set|-2}} X_{j_1}) \Big)^{\lfl \frac{m+1}{2} \rfl}. 
\end{equation*}
For the remaining points $u^*$ and $j_{|\set|-1}$, we perform the same operation as in Case 1 above. This again raises the degrees of all points in $\set$ to $m$ or $m+1$ with a desired bound. 

\bigskip

Every cluster $\set \in \uU$ falls into one of the above three cases. The graph $\Enh(\Gamma)$ is obtained by performing the above operations to all $\set \in \uU$. It is clear from the bounds in the above three situations that there exists $C>0$ such that
\begin{equation*}
|\Gamma| \leq C |\Enh(\Gamma)|. 
\end{equation*}
It is also straightforward to check that in $\Enh(\Gamma)$, we have $\deg(x_j) \in \{m,m+1\}$ for all $j \in [K]$, and hence it represents one of the pairings from the expectation
\begin{equation*}
\E \prod_{j=1}^{K} \big( X_{j}^{\diamond m} + X_{j+1}^{\diamond (m+1)} \big). 
\end{equation*}
Hence, we deduce there exists $C>0$ such that
\begin{equation} \label{eq:enhance_bound}
|\Gamma| \leq C |\Enh(\Gamma)| \leq C \E \prod_{j=1}^{K} \big( X_{j}^{\diamond m} + X_{j+1}^{\diamond (m+1)} \big)
\end{equation}
for all $\theta$, $\eps$ and $\x$, and this is true for all $\Gamma \in \Omega^*$. Combining \eqref{eq:enhance_bound} and Proposition~\ref{pr:simplified}, we obtain the bound \eqref{eq:main} in the case $|\sS| \geq 1$. 

Since the bound when $\sS=\emptyset$ has already been established in \eqref{eq:multi}, we have thus completed the proof of Theorem~\ref{th:main}.

\section{Convergence of the fields -- proof of Theorem~\ref{th:convergence}}

We are now ready to prove Theorem~\ref{th:convergence}. For notational simplicity, we write $A \lesssim_{\mathbb{\alpha}} B$ to denote that $A \leq C B$, where the constant $C$ depends only on the parameter(s) in the subscripts of the symbol $\lesssim$ (and in this case $\alpha$). 

In order to apply the bound in Theorem~\ref{th:main}, we use the convention for $\hF$ such that
\begin{equation*}
F(x) = \int_{\R} \hF(\theta) e^{i \theta x} d \theta. 
\end{equation*}
But this only appears in intermediate steps, and the final statement does not depend on the definition of $\hF$. For every $\varphi: \R^d \rightarrow \R$, every $x \in \R^d$ and every $\lambda>0$, we let
\begin{equation*}
\varphi_{x}^{\lambda}(y) = \lambda^{-|\fs|} \varphi \Big( \frac{y_1-x_1}{\lambda^{s_1}}, \dots, \frac{y_d - x_d}{\lambda^{s_d}} \Big). 
\end{equation*}
Recall that $\Psi_\eps = \rho_\eps * \Psi$, where $\rho_\eps$ is the rescaled mollifier. Also recall the form of $a_m$ in \eqref{eq:a_m}. We first give the convergence criterion. 

\begin{prop} \label{pr:criterion}
	Let $\kappa>0$ and $m < \frac{|\fs|}{\alpha}$. If for every compact $\kK \subset \R^d$ and every $n \in \N$, we have
	\begin{equation} \label{eq:criterion}
	\sup_{\lambda \in (\eps, 1)} \sup_{x \in \kK} \sup_{\stackrel{\varphi \in \cC_{c}^{\infty}(\kK):}{\|\varphi\|_{\cC^{\frac{m \alpha}{2}}} \leq 1}} \lambda^{\frac{m \alpha}{2}+\kappa} \Big( \E |\scal{\eps^{-\frac{m \alpha}{2}} \hH_{m} \big(F(\eps^{\frac{\alpha}{2}} \Psi_{\eps}) \big) - a_{m} \Psi^{\diamond m}, \varphi_x^\lambda}|^{2n} \Big)^{\frac{1}{2n}} \rightarrow 0
	\end{equation}
	as $\eps \rightarrow 0$, then $\eps^{-\frac{m \alpha}{2}} \hH_{m}\big(F(\eps^{\frac{\alpha}{2}} \Psi_\eps)\big) \rightarrow a_m \Psi^{\diamond m}$ in $\cC^{-\frac{m\alpha}{2}-\kappa'}$ for every $\kappa' > \kappa$. 
\end{prop}

The proof of the proposition is standard Kolmogorov's criterion, and so it remains to prove \eqref{eq:criterion}. By stationarity, we can simply restrict to the case $x=0$ in \eqref{eq:criterion}. Writing $\|\cdot\|_{2n} := \big(\E |\cdot|^{2n} \big)^{\frac{1}{2n}}$ as well as $\Phi_\eps = \eps^{\frac{\alpha}{2}} \Psi_\eps$, we need to show for all small $\kappa$ that
\begin{equation} \label{eq:convergence_criterion}
\lambda^{\frac{m \alpha}{2}+\kappa} \|\scal{\eps^{-\frac{m \alpha}{2}} \hH_{m} \big(F(\Phi_\eps)\big)- a_{m} \Psi^{\diamond m}, \varphi^{\lambda}}\|_{2n} \rightarrow 0
\end{equation}
as $\eps \rightarrow 0$, uniformly over $\lambda \in (\eps, 1)$ and smooth $\varphi$ supported in a ball of radius $1$ such that $\|\varphi\|_{\cC^{\frac{m \alpha}{2}}} \leq 1$. The rest of the section is devoted to the proof of \eqref{eq:convergence_criterion}. 

Since $\Psi$ is stationary, so is $\Psi_{\eps} = \rho_{\eps} * \Psi$. Hence, $\Phi_{\eps}$ has stationary Gaussian distribution $\mu_{\eps} \sim \nN(0,\sigma_\eps^2)$ with
\begin{equation} \label{eq:sigma2_eps}
\sigma_{\eps}^{2} = \E \Phi_{\eps}^{2} = \eps^{\alpha} \int G(x-y) \rho_{\eps}(x) \rho_{\eps}(y) dx dy, 
\end{equation}
where $G$ is the correlation function of $\Psi$ as in Assumption~\ref{as:field}. The coefficient of the $m$-th term in the chaos expansion of $F(\Phi_\eps)$ is given by
\begin{equation*}
a_{m}^{(\eps)} = \frac{1}{m!} \big( F^{(m)}*\mu_{\eps}  \big)(0). 
\end{equation*}
We split the difference $\eps^{-\frac{m \alpha}{2}} \hH_{m}\big(F(\Phi_{\eps})\big) - a_m \Psi^{\diamond m}$ into three parts by
\begin{equation} \label{eq:separate}
\begin{split}
&\phantom{111}\eps^{-\frac{m \alpha}{2}} \hH_{m}\big(F(\Phi_{\eps})\big) - a_{m} \Psi^{\diamond m}\\
&= \Big( \eps^{-\frac{m \alpha}{2}} \hH_{m}\big(F(\Phi_{\eps})\big) - a_{m}^{(\eps)} \Psi_{\eps}^{\diamond m} \Big) + a_{m}^{(\eps)} \big( \Psi_{\eps}^{\diamond m} - \Psi^{\diamond m} \big) + \big( a_{m}^{(\eps)} - a_{m} \big) \Psi^{\diamond m}, 
\end{split}
\end{equation}
and we show that each of them satisfies a bound of the form of \eqref{eq:convergence_criterion}. 

The latter two terms are simpler. For the second one, we notice that $m < \frac{|\fs|}{\alpha}$ ensures $\Psi^{\diamond m}$ is well-defined, and for all sufficiently small $\kappa$, we have
\begin{equation*}
\|\scal{\Psi_\eps^{\diamond m} - \Psi^{\diamond m}, \varphi^{\lambda}}\|_{2n} \lesssim_{n} \eps^{\kappa} \lambda^{-\frac{m \alpha}{2}-\kappa}
\end{equation*}
uniformly over $\lambda \in (\eps,1)$. Also, since $a_{m}^{(\eps)}$ is uniformly bounded in $\eps$, it then follows immediately that
\begin{equation*}
\lambda^{\frac{m \alpha}{2}+\kappa} \|\scal{\Psi_{\eps}^{\diamond m} - \Psi^{\diamond m}, \varphi^{\lambda}}\|_{2n} \rightarrow 0
\end{equation*}
as $\eps \rightarrow 0$, uniformly over $\lambda \in (\eps, 1)$ and $\varphi$ in the range required in Proposition~\ref{pr:criterion}. 

For the third term, it suffices to notice that the assumption \eqref{eq:G_converge} on $G$ guarantees that $\sigma_{\eps}^{2} \rightarrow \sigma^2$, where $\sigma_\eps^2$ and $\sigma^2$ are given by \eqref{eq:sigma2_eps} and \eqref{eq:sigma2}. Hence, we immediately have $a_{m}^{(\eps)} \rightarrow a_m$. The desired bound of the form \eqref{eq:convergence_criterion} then follows immediately from the boundedness of $\Psi^{\diamond m}$ in $\cC^{-\frac{m \alpha}{2}-\kappa}$. 

We now turn to the first term on the right hand side of \eqref{eq:separate}, which requires the use of the bound in Theorem~\ref{th:main}. We first note that the covariance of $\Phi_\eps = \eps^{\frac{\alpha}{2}} \rho_{\eps} * \Psi$ has the form
\begin{equation*}
\E \big( \Phi_\eps(x) \Phi_\eps(y) \big) = (\rho_{\eps}^{\star 2} \star G)(x-y), 
\end{equation*}
where $\star$ denotes the forward convolution in the sense that $(f \star g)(x) = \int f(x+y) g(y) dy$. Assumption~\ref{as:field} on $G$ guarantees that $\Phi_{\eps}$ satisfies the assumption \eqref{eq:field_eps} in Theorem~\ref{th:main} with some $\Lambda>1$. Since $\Psi_{\eps}^{\diamond m} = \eps^{-\frac{m\alpha}{2}} \Phi_\eps^{\diamond m}$, and $a_{m}^{(\eps)}$ is precisely the $m$-th coefficient in the chaos expansion of $F(\Phi_{\eps})$, we have
\begin{equation} \label{eq:removal}
\eps^{-\frac{m \alpha}{2}} \hH_{m}\big(F(\Phi_{\eps})\big) - a_{m}^{(\eps)} \Psi_{\eps}^{\diamond m} = \eps^{-\frac{m \alpha}{2}} \hH_{m+1}\big(F(\Phi_{\eps})\big). 
\end{equation}
We leave aside the factor $\eps^{-\frac{m \alpha}{2}}$ and focus on $\hH_{m+1} \big( F(\Phi_\eps) \big)$ for a moment. Fourier expanding $F$ and changing the order of integration, we get the identity
\begin{equation*}
\scal{\hH_{m+1}\big( F(\Phi_\eps) \big), \varphi^{\lambda}} = \scal{\hF, \aA \Phi_\eps}_{\theta}, 
\end{equation*}
where
\begin{equation} \label{eq:A}
(\aA \Phi_\eps)(\theta) = (\aA_{\varphi, \lambda,m} \Phi_\eps)(\theta) = \int_{\R^d} \hH_{m+1} \big(e^{i \theta \Phi_{\eps}(x)} \big) \varphi^{\lambda}(x) dx, 
\end{equation}
and the subscript $\theta$ on the inner product indicates that the testing is taken with respect to the Fourier variable $\theta \in \R$. We now omit the subscripts in $\aA$ for simplicity. Recall that $\iI_k = (k-1,k+1)$. Multiplying $\aA \Phi_{\eps}$ by a partition of unity subordinate to the intervals $\{\iI_k\}_{k \in \Z}$ and separating these terms, we get the bound
\begin{equation*}
|\scal{\hH_{m+1}\big( F(\Phi_\eps) \big), \varphi^{\lambda}}| \leq C_M \sum_{k \in \Z} \|\hF\|_{M,\iI_k} \sup_{0 \leq r \leq M} \sup_{\theta \in \iI_k} |(\aA \Phi_\eps)^{(r)}(\theta)|, 
\end{equation*}
where $M \in \N$ is as in Assumption~\ref{as:F}. Now, taking $2n$-th moments on both sides and using triangle inequality, we get
\begin{equation} \label{eq:moment_1}
\|\scal{\hH_{m+1}\big( F(\Phi_\eps) \big), \varphi^{\lambda}}\|_{2n} \leq C_M \sum_{k \in \Z} \|\hF\|_{M,\iI_k}  \Big( \E \sup_{0 \leq r \leq M} \sup_{\theta \in \iI_k} |(\aA \Phi_\eps)^{(r)}(\theta)|^{2n} \Big)^{\frac{1}{2n}}. 
\end{equation}
There are two suprema inside the expectation. The first supremum is taken over $M+1$ elements, so we can move it out of $\big( \E |\cdot|^{2n} \big)^{\frac{1}{2n}}$ at the cost of a constant multiple depending on $M$ and $n$ only. The second one is taken over an interval, so we need the following lemma to interchange it with the expectation. 

\begin{lem} \label{le:exchange}
	Suppose $f$ is a random $\cC^1$ function on an interval $\iI$. For every $p \geq 1$, there exists $C$ depending on $p$ and $|\iI|$ only such that
	\begin{equation*}
	\E \sup_{\theta \in \iI} |f(\theta)|^{p} \leq C_{p,|\iI|} \sup_{\theta \in \iI} \E \Big( |f(\theta)|^{p} + |f'(\theta)|^{p} \Big). 
	\end{equation*}
\end{lem}
\begin{proof}
	Fix $\theta_0 \in \iI$ arbitrary. By fundamental theorem of calculus and H\"older's inequality, we have
	\begin{equation*}
	|f(\theta)| \leq |f(\theta_0)| + |\iI|^{1-\frac{1}{p}} \Big( \int_{\iI} |f'(x)|^{p} dx \Big)^{\frac{1}{p}}. 
	\end{equation*}
	Raising both sides to $p$-th power, we get
	\begin{equation*}
	\sup_{\theta} |f(\theta)|^{p} \leq C \Big( |f(\theta_0)|^{p} + \int_{\iI} |f'(\theta)|^{p} d\theta \Big), 
	\end{equation*}
	where $C$ depends on $p$ and $\iI$. The assertion then follows by taking expectation on both sides and noting that
	\begin{equation*}
	\E \int_{\iI} |f'(\theta)|^{p} d \theta = \int_{\iI} \E |f'(\theta)|^{p} d \theta \leq |\iI| \sup_{\theta \in \iI} \E |f'(\theta)|^{p}. 
	\end{equation*}
	This completes the proof of the lemma. 
\end{proof}

Using Lemma~\ref{le:exchange} to interchange the expectation and supremum, we have
\begin{equation*}
\E \sup_{0 \leq r \leq M} \sup_{\theta \in \iI_k} |(\aA \Phi_\eps)^{(r)}(\theta)|^{2n} \lesssim_{n,M} \sup_{0 \leq r \leq M+1} \sup_{\theta \in \iI_k} \E |(\aA \Phi_\eps)^{(r)}(\theta)|^{2n}, 
\end{equation*}
where the supremum over $r$ is taken over $r \leq M+1$ to include the one additional derivative required in the interchange. Plugging it back into the right hand side of \eqref{eq:moment_1}, we obtain
\begin{equation} \label{eq:moment_2}
\|\scal{\hH_{m+1}\big(F(\Phi_\eps)\big), \varphi^\lambda}\|_{2n} \lesssim_{n,M} \sum_{k \in \Z} \|\hF\|_{M,\iI_k} \sup_{0 \leq r \leq M+1} \sup_{\theta \in \iI_k} \Big( \E |(\aA \Phi_\eps)^{(r)}(\theta)|^{2n} \Big)^{\frac{1}{2n}}. 
\end{equation}
It then remains to control the quantity $\E |(\aA \Phi_\eps)^{(r)}(\theta)|^{2n}$. Recalling the expression of $\aA \Phi_\eps$ in \eqref{eq:A}, we have
\begin{equation*}
\E |(\aA \Phi_\eps)^{(r)}(\theta)|^{2n} = \int_{(\R^d)^{2n}}  \Big(  \E\prod_{j=1}^{2n} \d_{\theta}^{r} \hH_{m+1}\big( e^{i \theta \Phi_{\eps}(x_j)} \big) \Big) \cdot \Big( \prod_{j=1}^{2n} \varphi^{\lambda}(x_j) \Big) d\x, 
\end{equation*}
where we used the shorthand notation $\x=(x_1, \dots, x_{2n})$. We now apply Theorem~\ref{th:main} to the expectation part above, so that we get
\begin{equation*}
\E\prod_{j=1}^{2n} \d_{\theta}^{r} \hH_{m+1}\big( e^{i \theta \Phi_{\eps}(x_j)} \big) \lesssim_{r} \E \prod_{j=1}^{2n} \Big( \Phi_{\eps}^{\diamond (m+1)}(x_j) + \Phi_{\eps}^{\diamond (m+2)}(x_j) \Big). 
\end{equation*}
Plugging it into the integral on the right hand side above and using the identity
\begin{equation*}
\begin{split}
&\phantom{11}\int \bigg[ \E \prod_{j=1}^{2n} \Big( \Phi_{\eps}^{\diamond (m+1)}(x_j) + \Phi_{\eps}^{\diamond (m+2)}(x_j) \Big) \bigg] \cdot \Big( \prod_{j=1}^{2n} \varphi^{\lambda}(x_j) \Big) d\x\\
&= \E |\scal{\Phi_\eps^{\diamond (m+1)}+\Phi_\eps^{\diamond (m+2)}, \varphi^\lambda}|^{2n}, 
\end{split}
\end{equation*}
we get
\begin{equation} \label{eq:moment_3}
\big( \E |(\aA \Phi_\eps)^{(r)}(\theta)|^{2n} \big)^{\frac{1}{2n}} \lesssim_{r,n} \sum_{\ell=1}^{2} \big( \E |\scal{\Phi_\eps^{\diamond (m+\ell)}, \varphi^\lambda}|^{2n}\big)^{\frac{1}{2n}}. 
\end{equation}
In particular, the bound is uniform in both $\eps$ and $\theta$. We have the following lemma controlling the right hand side above.

\begin{lem} \label{le:higher_chaos}
	For every integer $\ell \geq 1$ and every sufficiently small $\kappa$, we have
	\begin{equation*}
	\big( \E |\scal{\Phi_\eps^{\diamond (m+\ell)}, \varphi^\lambda}|^{2n}\big)^{\frac{1}{2n}} \lesssim_{n,\ell} \eps^{\frac{m \alpha}{2}+\kappa} \lambda^{-\frac{m \alpha}{2}-\kappa}
	\end{equation*}
	uniformly over $\eps, \lambda \in (0,1)$. 
\end{lem}
\begin{proof}
	Since $\Phi_\eps$ is Gaussian, by equivalence of moments, the left hand side above can be controlled by
	\begin{equation*}
	\big( \E |\scal{\Phi_\eps^{\diamond (m+\ell)}, \varphi^\lambda}|^{2n}\big)^{\frac{1}{2n}} \lesssim_{n} \big( \E |\scal{\Phi_\eps^{\diamond (m+\ell)}, \varphi^\lambda}|^{2}\big)^{\frac{1}{2}}, 
	\end{equation*}
	so we only need to bound the second moment. By Wick's formula, we have
	\begin{equation*}
	\E |\scal{\Phi_\eps^{\diamond (m+\ell)}, \varphi^\lambda}|^{2} = (m+\ell)! \int \big( \E \Phi_{\eps}(x) \Phi_{\eps}(y) \big)^{m+\ell} \varphi^{\lambda}(x) \varphi^{\lambda}(y) dx dy. 
	\end{equation*}
	Since $\ell \geq 1$, we have
	\begin{equation*}
	\big( \E \Phi_\eps(x) \Phi_\eps(y) \big)^{m+\ell} \leq \frac{\Lambda^{m+\ell} \eps^{\alpha(m+\ell)}}{(|x-y|+\eps)^{\alpha(m+\ell)}} \leq \frac{\Lambda^{m+\ell} \eps^{m \alpha+\kappa}}{|x-y|^{m \alpha + \kappa}}
	\end{equation*}
	for all $\kappa \in (0,\ell \alpha)$. For all $\kappa$ sufficiently small such that $m \alpha + \kappa < |\fs|$, the singularity on the right hand side above is integrable, so we have
	\begin{equation*}
	\E |\scal{\Phi_\eps^{\diamond (m+\ell)}, \varphi^\lambda}|^{2} \lesssim \eps^{m \alpha + \kappa} \int \frac{\varphi^{\lambda}(x) \varphi^{\lambda}(y)}{|x-y|^{m\alpha+\kappa}} dx dy \lesssim \eps^{m \alpha+\kappa} \lambda^{-m\alpha-\kappa}. 
	\end{equation*}
	The proof is complete by taking square roots on both sides and replacing $\kappa$ by $2\kappa$. 
\end{proof}

Now, combining \eqref{eq:moment_2}, \eqref{eq:moment_3} and Lemma~\ref{le:higher_chaos}, and using Assumption~\ref{as:F} on $F$, we get
\begin{equation*}
\begin{split}
\|\scal{\hH_{m+1}\big(F(\Phi_\eps)\big), \varphi^\lambda}\|_{2n} &\lesssim_{n,M} \eps^{\frac{m\alpha}{2}+\kappa} \lambda^{-\frac{m\alpha}{2}-\kappa} \sum_{k \in \Z} \|\hF\|_{M,\iI_k}\\
&\lesssim_{n,M} \eps^{\frac{m\alpha}{2}+\kappa} \lambda^{-\frac{m\alpha}{2}-\kappa}. 
\end{split}
\end{equation*}
Substituting the above bound back to \eqref{eq:removal}, we deduce that
\begin{equation*}
\lambda^{\frac{m\alpha}{2}+\kappa}\|\scal{ \eps^{-\frac{m \alpha}{2}} \hH_{m}\big(F(\Phi_{\eps})\big) - a_{m}^{(\eps)} \Psi_{\eps}^{\diamond m}, \varphi^\lambda}\|_{2n} \lesssim_{n,M} \eps^{\kappa}, 
\end{equation*}
which is the desired bound. The proof of Theorem~\ref{th:convergence} is thus complete. 

\bibliographystyle{Martin}
\bibliography{Refs}

\end{document}